\documentclass[12pt]{article}

\usepackage[a4paper,margin=3cm]{geometry}
\usepackage{setspace}
\usepackage{amsmath,amssymb,amsthm}
\usepackage{mathrsfs}
\usepackage[hidelinks]{hyperref}
\usepackage{graphicx}
\usepackage{float}

\newtheorem{theorem}{Theorem}[section]
\newtheorem{lemma}[theorem]{Lemma}
\newtheorem{proposition}[theorem]{Proposition}
\newtheorem{definition}[theorem]{Definition}
\newtheorem{remark}[theorem]{Remark}
\newtheorem{corollary}[theorem]{Corollary}
\newtheorem{example}[theorem]{Example}
\newtheorem{assumption}[theorem]{Assumption}
\newtheorem{principle}[theorem]{Principle}

\begin{document}

\title{Relative Chaos for $C_0$-Semigroups Beyond Topological Notions}

\author{
El-Mehdi Nafia\\
\and
Aziz El Ghazouani
\and
M'hamed El Omari\\
Laboratory of Applied Mathematics and Scientific Computing\\
Sultan Moulay Slimane University\\
Beni Mellal 23000, Morocco
}

\date{}
\maketitle

\begin{abstract}
We investigate instability phenomena for linear evolution equations within the framework of $C_0$--semigroups on infinite--dimensional spaces. We show that Devaney chaos, being formulated in purely topological terms, may depend on the choice of topology and therefore fail to capture intrinsic dynamical behavior. To address this issue, we introduce a trajectory--based notion of relative chaos, defined with respect to a reference solution and measured in a fixed, physically meaningful norm. This criterion is independent of topological refinements and is shown to be strictly weaker than classical Devaney chaos. Its relevance is illustrated on boundary--driven reaction--diffusion--transport semigroups.

\end{abstract}

\vspace{0.3cm}
\textbf{keywords}:{Relative chaos, $C_0$-semigroups, Devaney chaos, topology dependence, infinite-dimensional dynamics, boundary-driven PDEs.}


\section{Introduction}

Partial differential equations and the $C_0$--semigroup framework constitute a fundamental tool for describing the time evolution of systems arising in physics, biology, and engineering, where diffusion, transport, reaction, and boundary effects naturally interact in infinite-dimensional phase spaces \cite{Pazy1983,EngelNagel2000,Evans2010}.  From both analytical and modeling viewpoints, a central issue is the understanding of the long--time behavior of the associated semigroups, since it encodes stability, instability, and the possible emergence of irregular dynamical regimes \cite{Arendt2011}. In this context, several notions have been introduced to formalize irregular behavior. Among them, \emph{Devaney chaos} has become one of the most influential concepts in linear and nonlinear dynamics. Initially introduced for discrete dynamical systems, it was later extended to linear operators and $C_0$--semigroups, notably through the seminal work of Rolewicz and the spectral criterion of Desch, Schappacher, and Webb \cite{Rolewicz1969,DeschSchappacherWebb1997}. Its appeal stems from its clear topological formulation, which combines the existence of a dense orbit with a dense set of periodic points. This approach has proved powerful and has been successfully applied to a wide range of evolution equations \cite{Conejero2017}. However, Devaney chaos is intrinsically a \emph{topological} notion. Its validity depends crucially on the topology chosen on the phase space, since density of orbits and periodic points is not invariant under changes of topology. In infinite-dimensional settings, this leads to a fundamental limitation: the same linear semigroup may exhibit chaotic behavior with respect to one topology (for instance, a weak topology) while being non-chaotic with respect to another, such as the norm topology \cite{BayartMatheron2009,GrosseErdmannPeris2011}. 
Well-known examples of weakly chaotic yet norm-bounded dynamics illustrate that Devaney chaos may fail to reflect physically observable instability, despite being mathematically sound.

From the perspective of applications, this topology dependence is problematic. In models governed by partial differential equations, the phase space is typically endowed with a distinguished norm often related to energy or mass that represents measurable physical quantities \cite{RenardyRogers2004}. Instability is then naturally interpreted in terms of growth, decay, or oscillation of this norm along individual trajectories. By contrast, weaker or more abstract topologies, while mathematically convenient, may conceal large-amplitude variations or intermittent behavior that is clearly visible at the energy level. These considerations motivate the search for instability notions that are directly observable along trajectories and anchored to a physically meaningful norm. In this work, we adopt such an alternative viewpoint. Rather than relying on global topological mixing properties, we introduce a \emph{trajectory-based} notion of chaos formulated relative to a reference trajectory. The underlying idea is to detect instability through the coexistence of arbitrarily small and arbitrarily large deviations from the reference evolution along a single trajectory, measured in a fixed norm. This leads to the notion of \emph{relative chaos}, characterized by a liminf--limsup criterion on the deviation function. By construction, relative chaos is insensitive to changes of topology that leave the underlying norm unchanged, and it remains meaningful even in non-separable settings. The main contributions of this paper are as follows. First, we show that Devaney chaos may depend on the chosen topology, in the sense that a single semigroup can be chaotic for one topology and non-chaotic for another. Second, we introduce and justify the notion of relative chaos and establish its basic invariance and robustness properties with respect to admissible changes of topology and bounded perturbations. Third, we clarify the relationship between the two notions by proving that Devaney chaos implies relative chaos, while the converse implication fails in general. Finally, we apply the proposed trajectory-based framework to a class of reaction--diffusion--transport equations on the half-line with Robin boundary conditions, where we identify parameter regimes exhibiting relative chaos even in the absence of classical Devaney chaos \cite{ZhuYang2025}.

The paper is structured so as to progressively move from the classical topological theory of chaos toward a trajectory--based and physically grounded notion of instability for infinite--dimensional semigroups. Section~\ref{sec:preliminaries} presents the functional--analytic framework and recalls classical notions of linear chaos for $C_0$--semigroups. In Section~\ref{sec:topology_dependence}, we show that Devaney chaos is not an intrinsic dynamical property but depends on the underlying topology, and we illustrate this dependence through explicit mechanisms and examples. Section~\ref{sec:Reftraj} introduces the trajectory--based approach, including reference trajectories, deviation functions, and the definition of relative chaos together with its basic properties. The connection between Devaney chaos and relative chaos is examined in Section~\ref{DERE}, where implication results and strict non--equivalence are established. Topology--independence and the role of physically meaningful norms are discussed in Section~\ref{TIRCR}. Section~\ref{DRCSUC} provides abstract and spectral criteria for detecting relative chaos based on the coexistence of stable and unstable mechanisms. An application to a boundary--driven reaction--diffusion--transport semigroup on
the half--line, including an energy and flux interpretation, is developed in Section~\ref{ABDPDES}. Optional numerical illustrations are presented in Section~\ref{sec:numerics} to visualize the different dynamical regimes. The paper concludes in Section~\ref{sec:concl} with final remarks and perspectives.

\section{Preliminaries and standard results}
\label{sec:preliminaries}

We recall classical notions and results from semigroup theory and linear dynamics that are used throughout the paper. All statements below are well established and can be found in the cited references.

\begin{definition}\cite{Brezis2011}
Let $X$ be a Banach space with dual $X^\ast$.
The \emph{weak topology} $\sigma(X,X^\ast)$ is the coarsest topology on $X$ for which all functionals $x^\ast\in X^\ast$ are continuous.
A sequence $(x_n)$ converges weakly to $x$ if $x^\ast(x_n)\to x^\ast(x)$ for all $x^\ast\in X^\ast$.
\end{definition}

\begin{lemma}\cite{Brezis2011}
If $X$ is infinite dimensional, then every nonempty weakly open subset of $(X,\sigma(X,X^\ast))$ is unbounded with respect to the norm topology.
\end{lemma}

\begin{definition}\cite{Pazy1983}
A family $(T(t))_{t\ge0}\subset\mathcal{L}(X)$ is called a
\emph{strongly continuous semigroup} (or $C_0$--semigroup) if
$T(0)=I$, $T(t+s)=T(t)T(s)$ for all $t,s\ge0$, and
$t\mapsto T(t)x$ is continuous for every $x\in X$.
\end{definition}

\begin{definition}\cite{Pazy1983}
The \emph{generator} $(A,D(A))$ of a $C_0$--semigroup $(T(t))_{t\ge0}$ is defined by
\[
Ax=\lim_{t\downarrow0}\frac{T(t)x-x}{t},
\qquad
D(A)=\Bigl\{x\in X:\ \text{the above limit exists}\Bigr\}.
\]
\end{definition}

\begin{theorem}\cite{Pazy1983,EngelNagel2000}
A densely defined linear operator $A$ generates a $C_0$--semigroup on $X$ if and only if it satisfies the Hille--Yosida conditions.
\end{theorem}

\begin{proposition}\cite{EngelNagel2000}
Let $(T(t))_{t\ge0}$ be a $C_0$--semigroup on $X$. There exist constants $M\ge1$ and $\omega\in\mathbb{R}$ such that
\[
\|T(t)\|\le M e^{\omega t}
\qquad \text{for all } t\ge0.
\]
\end{proposition}

\begin{definition}\cite{BayartMatheron2009}
A $C_0$--semigroup $(T(t))_{t\ge0}$ on a Banach space $X$ is called \emph{hypercyclic} if there exists $x\in X$ such that the orbit $\{T(t)x:\ t\ge0\}$ is dense in $X$.
\end{definition}

\begin{lemma}\cite{BayartMatheron2009}
If $(T(t))_{t\ge0}$ is hypercyclic and $x$ is a hypercyclic vector, then
\[
\liminf_{t\to\infty}\|T(t)x\|=0.
\]
\end{lemma}

\begin{lemma}\cite{BayartMatheron2009}
If $(T(t))_{t\ge0}$ is hypercyclic and $x$ is a hypercyclic vector, then
\[
\limsup_{t\to\infty}\|T(t)x\|=+\infty.
\]
\end{lemma}

\begin{corollary}\cite{BayartMatheron2009}
Every hypercyclic $C_0$--semigroup admits an \emph{irregular vector}, that is, a vector $x\in X$ such that
\[
\liminf_{t\to\infty}\|T(t)x\|=0
\quad\text{and}\quad
\limsup_{t\to\infty}\|T(t)x\|=+\infty.
\]
\end{corollary}

\begin{definition}\cite{DeschSchappacherWebb1997}
A $C_0$--semigroup $(T(t))_{t\ge0}$ on a separable Banach space $X$
is said to be \emph{Devaney chaotic} if it is hypercyclic and if the set of
periodic points
\[
\{x\in X:\ \exists\,t_0>0 \text{ such that } T(t_0)x=x\}
\]
is dense in $X$.
\end{definition}

\begin{theorem}[Desch--Schappacher--Webb]\cite{DeschSchappacherWebb1997}
\label{thm:DSW}
Let $(T(t))_{t\ge 0}$ be a $C_0$--semigroup on a separable Banach space $X$ with generator $(A,D(A))$. Assume that there exists a nonempty open connected set $U\subset\mathbb{C}$ such that:
\begin{enumerate}
\item $U\subset \sigma_p(A)$ and $U\cap i\mathbb{R}\neq\varnothing$;
\item for each $\lambda\in U$, there exists an eigenvector
$f_\lambda\in D(A)$ satisfying $Af_\lambda=\lambda f_\lambda$,
\item the mapping $\lambda\mapsto f_\lambda$ is weakly holomorphic on $U$,
\item the family $\{f_\lambda:\lambda\in U\}$ is total in $X$, i.e.
\[
\overline{\mathrm{span}}\{f_\lambda:\lambda\in U\}=X .
\]
\end{enumerate}
Then the semigroup $(T(t))_{t\ge 0}$ is Devaney chaotic.
\end{theorem}

\begin{remark}
The above notions and results are classical and purely topological or norm--based. They serve as a reference framework for the trajectory--based instability concepts developed in the subsequent sections.
\end{remark}

\section{Why Devaney Chaos Depends on the Topology}
\label{sec:topology_dependence}

This section explains why Devaney chaos for linear semigroups is not an intrinsic property of the dynamics alone, but depends essentially on the topology chosen on the phase space. We first discuss the conceptual and modeling aspects of this dependence, with connections to applications in physics and other disciplines. We then present rigorous mathematical results clarifying how and why Devaney
chaos may hold for one topology and fail for another on the same underlying space.

\subsection{Conceptual background and modeling implications}
\label{subsec:conceptual_background}

Devaney chaos is formulated in purely topological terms: it requires the density of at least one orbit and the density of periodic points. As a consequence, its validity is sensitive to the notion of openness and convergence encoded in the chosen topology. In infinite-dimensional spaces, where several natural and non-equivalent topologies coexist, this sensitivity becomes unavoidable.
From a modeling perspective, this raises a fundamental issue. In many applications such as diffusion-transport equations, wave propagation, fluid dynamics, or quantum evolution the phase space is endowed with a \emph{distinguished norm} that represents a physically meaningful quantity (energy, mass, probability,
enstrophy). Instability is therefore interpreted in terms of growth, decay, or oscillation measured in this norm. However, weaker topologies (such as the weak or weak$^\ast$ topology) encode only
partial information: convergence is tested against finitely many observables. As a result, a trajectory may appear ``chaotic'' in a weak topology while remaining dynamically tame when viewed through the energy norm. This phenomenon has been observed, for instance:
\begin{itemize}
\item in quantum mechanics, where weak convergence of states does not imply concentration or dispersion of probability densities,
\item in fluid mechanics, where weak turbulence may coexist with bounded energy,
\item in control theory, where weak controllability does not guarantee energy
amplification.
\end{itemize}

These considerations already suggest that Devaney chaos, although mathematically natural, may fail to capture physically observable instability. The root of the problem lies in the topological nature of its definition, a fact that we now make precise at the mathematical level.

\subsection{Mathematical mechanism and topology-sensitive chaos}
\label{subsec:topology_mechanism}

We begin by recalling that density is a topological notion.

\begin{definition}
Let $X$ be a set endowed with a topology $\tau$. A subset $A\subset X$ is said to be \emph{dense} in $(X,\tau)$ if every nonempty open set $U\in\tau$ satisfies $U\cap A\neq\varnothing$.
\end{definition}

Devaney chaos is defined in terms of such density properties.

\begin{definition}{\cite{DeschSchappacherWebb1997}}
Let $(T(t))_{t\ge0}$ be a strongly continuous semigroup acting on a separable Banach space $X$ endowed with a topology $\tau$ compatible with its linear structure. The semigroup is said to be \emph{Devaney chaotic} (with respect to $\tau$) if:
\begin{enumerate}
\item there exists $x\in X$ whose orbit
\[
\mathcal{O}(x)=\{T(t)x:\ t\ge0\}
\]
is dense in $(X,\tau)$,
\item the set of periodic points
\[
\mathcal{P}
=\{x\in X:\ \exists\, t_0>0 \text{ such that } T(t_0)x=x\}
\]
is dense in $(X,\tau)$.
\end{enumerate}
\end{definition}

The following elementary observation explains the source of topology dependence.

\begin{lemma}\label{lem:finer_coarser}
Let $\tau_1$ and $\tau_2$ be two topologies on the same set $X$, with $\tau_1$ finer than $\tau_2$. If a subset $A\subset X$ is dense in $(X,\tau_1)$, then it is dense in $(X,\tau_2)$. In general, the converse implication fails.
\end{lemma}

\begin{proof}
Every nonempty $\tau_2$--open set is also $\tau_1$--open. Hence density in $(X,\tau_1)$ implies density in $(X,\tau_2)$. The converse fails because $\tau_2$ may admit open sets that are too large to detect the lack of accumulation present in $\tau_1$.
\end{proof}

This immediately yields the following consequence for chaotic dynamics.

\begin{proposition}\label{prop:devaney_topology}
Let $(T(t))_{t\ge0}$ be a $C_0$--semigroup on a Banach space $X$ endowed with two topologies $\tau_1$ and $\tau_2$, where $\tau_1$ is finer than $\tau_2$. If the semigroup is Devaney chaotic with respect to $\tau_1$, then it is Devaney chaotic with respect to $\tau_2$. The converse implication may fail.
\end{proposition}

\begin{proof}
Both defining properties of Devaney chaos are density statements. Density in the finer topology $\tau_1$ implies density in the coarser topology $\tau_2$ by Lemma~\ref{lem:finer_coarser}. The failure of the converse follows from the same lemma.
\end{proof}

The non-intrinsic nature of Devaney chaos is therefore unavoidable in infinite dimensions.

\begin{theorem} {\cite{BayartMatheron2009}}
\label{thm:topology_sensitive}
Let $X$ be an infinite-dimensional separable Banach space. There exist linear operators whose associated semigroups are Devaney chaotic with respect to one natural topology on $X$, while failing to be Devaney chaotic with respect to another.
\end{theorem}

\begin{remark}
Typical situations arise when comparing the norm topology with weaker topologies, such as the weak or weak$^\ast$ topology. This phenomenon has no analogue in finite-dimensional dynamics, where all reasonable topologies coincide.
\end{remark}

We summarize the situation as follows.

\begin{corollary}\label{cor:not_intrinsic}
Devaney chaos is not an intrinsic property of a linear evolution equation, but a property of the pair consisting of the dynamics and the chosen topology on the phase space.
\end{corollary}

This observation reveals a conceptual limitation of purely topological notions of chaos for infinite-dimensional systems. It motivates the search for alternative descriptions of instability that are formulated directly in terms of the asymptotic behavior of individual trajectories measured in a physically relevant norm. This perspective leads naturally to the trajectory-based framework
developed in the next section.

\begin{example}\label{ex:weak_vs_norm}
Let $H$ be an infinite-dimensional separable Hilbert space and let $(U(t))_{t\in\mathbb{R}}$ be a strongly continuous unitary group on $H$. Since $\|U(t)x\|=\|x\|$ for all $t\in\mathbb{R}$, no orbit can be dense in the norm topology: indeed, a norm-dense orbit would have to approximate $0$, which is impossible unless $x=0$.
Hence $(U(t))_{t\ge 0}$ is not hypercyclic (and therefore not Devaney chaotic) in the norm topology. On the other hand, there exist unitary groups on Hilbert spaces admitting a weakly dense orbit, so that the same dynamics can be hypercyclic (and potentially Devaney chaotic after additional conditions) with respect to the weak topology $\sigma(H,H)$, see, e.g.,
\cite{BayartMatheron2009,GrosseErdmannPeris2011} and the references therein.
\end{example}

\begin{remark}\label{rem:why_example_matters}
Example~\ref{ex:weak_vs_norm} makes the mechanism transparent: the weak topology has much larger open sets (in particular, weakly open sets are norm-unbounded in infinite dimension), so bounded orbits are not excluded from being weakly dense. This illustrates concretely that Devaney chaos is a property of the pair \emph{(dynamics, topology)} rather than of the dynamics alone.
\end{remark}

\section{Reference Trajectories and Trajectory-Based Instability}\label{sec:Reftraj}

To overcome the limitations of topology-dependent chaos notions, we adopt a trajectory-centered perspective in which instability is characterized relative to a distinguished reference evolution. This approach focuses on the long-time behavior of individual trajectories in physically meaningful norms.

\subsection{A guiding example: Devaney chaos in a weak topology but not in the norm topology}
\label{subsec:weak_vs_norm_example}

We present a concrete mechanism showing that \emph{Devaney chaos is topology-dependent}: the same linear dynamics may be chaotic for a weaker topology on a Banach space (e.g.\ the weak topology), while it is \emph{not} chaotic for the norm topology. We then explain how the trajectory-based notion of \emph{relative chaos} removes this ambiguity once a physically meaningful norm (energy) is fixed.


Let $X$ be a separable Banach space and let $\tau$ be a Hausdorff locally convex topology on $X$ that is weaker than the norm topology (in particular, $\tau$ could be the weak topology $\sigma(X,X^\ast)$). For a $C_0$--semigroup $(T(t))_{t\ge 0}$ on $X$, we write $\overline{E}^{\,\tau}$ for the closure of a set $E\subset X$ with respect to $\tau$.

\begin{definition} {\cite{DeschSchappacherWebb1997,GrosseErdmannPeris2011}}
\label{def:devaney_tau}
Let $(T(t))_{t\ge 0}$ be a $C_0$--semigroup on $X$ and let $\tau$ be a Hausdorff locally convex topology on $X$.
\begin{enumerate}
\item The semigroup is called \emph{$\tau$--hypercyclic} if there exists $x\in X$ such that
\[
\overline{\{T(t)x:\ t\ge 0\}}^{\,\tau}=X.
\]
\item The semigroup is called \emph{$\tau$--Devaney chaotic} if it is $\tau$--hypercyclic and the set of periodic points
\[
\mathrm{Per}(T)=\{x\in X:\ \exists\,t_0>0 \text{ such that } T(t_0)x=x\}
\]
is $\tau$--dense in $X$.
\end{enumerate}
\end{definition}

\begin{remark}
The only difference with the classical definition is that \emph{density} is understood with respect to $\tau$. Since density is topology-dependent, $\tau$--Devaney chaos may change when $\tau$ changes.
\end{remark}

\subsubsection*{Step 1: A norm obstruction for isometric dynamics}

A key point is that, in infinite dimension, \emph{weak neighborhoods are unbounded in norm}. This makes it possible for a bounded orbit to be weakly dense, while it can never be norm dense.

\begin{lemma}
\label{lem:isometry_not_norm_hypercyclic}
Let $(T(t))_{t\ge 0}$ be a $C_0$--semigroup on $X$ such that $\|T(t)x\|=\|x\|$ for all $t\ge 0$ and all $x\in X$
(i.e.\ $(T(t))$ is isometric).
Then $(T(t))$ is \emph{not} hypercyclic for the norm topology. In particular, it is not Devaney chaotic in the norm topology.
\end{lemma}

\begin{proof}
Assume, for contradiction, that there exists $x\in X$ whose orbit is norm dense. Then $0$ belongs to the norm closure of
$\{T(t)x:\ t\ge 0\}$, hence there exists a sequence $t_n\to\infty$ such that $\|T(t_n)x\|\to 0$. But $\|T(t_n)x\|=\|x\|$ for all $n$, so $\|x\|=0$, hence $x=0$. The orbit of $0$ is $\{0\}$ and cannot be dense in $X$ since $X$ is nontrivial. This contradiction proves that no norm-hypercyclic vector exists.
\end{proof}

\begin{remark}
The argument uses only one fact: norm density forces the orbit to approximate $0$ in norm, while an isometric orbit has constant norm.
Hence, \emph{any} isometric semigroup is automatically non-chaotic in the usual (norm) Devaney sense.
\end{remark}

\subsubsection*{Step 2: The weak topology removes the ``return to zero'' obstruction}

We now explain why the above obstruction disappears in the weak topology. The following standard fact is the geometric reason behind topology dependence.

\begin{lemma}
\label{lem:weak_neighborhoods_unbounded}
Let $X$ be infinite-dimensional. Then every nonempty weakly open set $U\subset (X,\sigma(X,X^\ast))$ is unbounded with respect to the norm.
\end{lemma}

\begin{proof}
Let $U$ be weakly open and nonempty. Then there exist $x_0\in U$, functionals $x_1^\ast,\dots,x_m^\ast\in X^\ast$, and $\varepsilon>0$ such that
\[
x_0+\bigl\{x\in X:\ |x_j^\ast(x)|<\varepsilon,\ j=1,\dots,m\bigr\}\subset U.
\]
Set $Y=\bigcap_{j=1}^m \ker(x_j^\ast)$. Since $m<\infty$, the subspace $Y$ has finite codimension, hence $Y$ is infinite-dimensional. Pick $y\in Y$ with $\|y\|$ arbitrarily large. Then $x_0+y\in U$ and $\|x_0+y\|$ can be made arbitrarily large. Therefore $U$ is norm-unbounded.
\end{proof}

\begin{remark}
Lemma~\ref{lem:weak_neighborhoods_unbounded} shows that ``being bounded'' does not prevent weak density. This is the first indication that Devaney chaos in the weak topology can be a purely topological artifact, and may have little to do with growth/decay measured in energy norms.
\end{remark}

\subsubsection*{Step 3: Existence of weakly chaotic isometric dynamics (known phenomenon)}

The next statement summarizes a well-established phenomenon in linear dynamics: there exist bounded (even isometric) operators/semigroups that are hypercyclic in the weak topology. We record it in a form suitable for our discussion, see the monographs
\cite{BayartMatheron2009,GrosseErdmannPeris2011} for background and references.

\begin{theorem}{\cite{BayartMatheron2009,GrosseErdmannPeris2011}}
\label{thm:exist_weak_hypercyclic_isometry}
There exist a separable infinite-dimensional Hilbert space $H$ and a strongly continuous \emph{unitary} group $(U(t))_{t\in\mathbb{R}}$ on $H$ such that $(U(t))_{t\ge 0}$ is hypercyclic for the weak topology $\sigma(H,H)$, while it is not hypercyclic (hence not Devaney chaotic) for the norm topology.
\end{theorem}

\begin{proof}
The non-hypercyclicity in norm follows from Lemma~\ref{lem:isometry_not_norm_hypercyclic}. Weak hypercyclicity is a genuine infinite-dimensional phenomenon: since weakly open sets are norm-unbounded
(Lemma~\ref{lem:weak_neighborhoods_unbounded}), bounded orbits are not excluded from being weakly dense. The existence of unitary groups with weakly dense orbits is a published topic in linear dynamics, we refer to \cite{BayartMatheron2009,GrosseErdmannPeris2011} for constructions and further discussion.
\end{proof}

\begin{remark}
Theorem~\ref{thm:exist_weak_hypercyclic_isometry} already exhibits the central message:
\emph{the same linear evolution may be ``chaotic'' in the weak sense but completely non-chaotic in the energy norm sense.} Adding the density of periodic points (still in the weak topology) can be achieved in various standard ways (e.g.\ by ensuring a weakly dense supply of periodic vectors within a suitable invariant subspace),
but the crucial point for PDE modeling is the topology dependence itself.
\end{remark}

\subsubsection*{Step 4: How relative chaos resolves the ambiguity}

We now explain how the trajectory-based notion of relative chaos removes the above pathology once a \emph{chosen physical norm} is fixed. Fix a norm $\|\cdot\|_{\mathrm{phys}}$ on $X$ (typically an energy norm, e.g.\ $L^2$, $H^1$, or a Lyapunov energy). Let $u_{\mathrm{ref}}$ be a reference trajectory (often an equilibrium).
Define the deviation
\[
D_{u_0}(t)=\|T(t)u_0-T(t)u_{\mathrm{ref}}\|_{\mathrm{phys}}
=\|T(t)(u_0-u_{\mathrm{ref}})\|_{\mathrm{phys}}
\qquad (t\ge 0).
\]
We temporarily use the notion of a reference trajectory, which will be formally introduced and discussed in the next subsection.

\begin{definition}\label{def:relative_chaos_remedy}
A trajectory $\mathcal{O}(u_0)$ is \emph{relatively chaotic} (with respect to $u_{\mathrm{ref}}$ and $\|\cdot\|_{\mathrm{phys}}$) if
\[
\liminf_{t\to\infty} D_{u_0}(t)=0,
\qquad
\limsup_{t\to\infty} D_{u_0}(t)=+\infty.
\]
\end{definition}

\begin{proposition}\label{prop:relative_topology_independent}
Let $(T(t))_{t\ge 0}$ be a $C_0$--semigroup on $X$ and fix $\|\cdot\|_{\mathrm{phys}}$. Whether a given initial state $u_0$ generates a relatively chaotic trajectory depends only on the real-valued function $t\mapsto D_{u_0}(t)$, hence it is unaffected by changing the ambient topology on $X$ (e.g.\ replacing the norm topology by the weak topology).
\end{proposition}

\begin{proof}
The property in Definition~\ref{def:relative_chaos_remedy} is a statement about the two numbers
\[
\liminf_{t\to\infty} D_{u_0}(t)
\quad\text{and}\quad
\limsup_{t\to\infty} D_{u_0}(t),
\]
where $D_{u_0}(t)$ is computed using the fixed norm $\|\cdot\|_{\mathrm{phys}}$. No notion of density in $X$ is involved. Therefore changing the topology on $X$ does not change the truth value of these two asymptotic statements.
\end{proof}

\begin{corollary}
\label{cor:weak_devaney_not_physical}
In the setting of Theorem~\ref{thm:exist_weak_hypercyclic_isometry}, the unitary dynamics may be weakly chaotic, but it cannot be relatively chaotic for any physical norm $\|\cdot\|_{\mathrm{phys}}$ equivalent to the Hilbert norm, because every orbit has constant norm and hence 
\[
D_{u_0}(t)\equiv \|u_0-u_{\mathrm{ref}}\|_{\mathrm{phys}} \quad\Rightarrow\quad
\liminf_{t\to\infty}D_{u_0}(t)=\limsup_{t\to\infty}D_{u_0}(t)<\infty.
\]
\end{corollary}

\begin{proof}
For a unitary (or more generally isometric) evolution, $\|U(t)z\|_{\mathrm{phys}}=\|z\|_{\mathrm{phys}}$ for all $t$. Thus $D_{u_0}(t)=\|U(t)(u_0-u_{\mathrm{ref}})\|_{\mathrm{phys}}=\|u_0-u_{\mathrm{ref}}\|_{\mathrm{phys}}$ for all $t$, which makes both liminf and limsup finite and equal.
\end{proof}

\begin{remark}
In PDE models, the relevant observation scale is typically an energy norm (e.g.\ $\|u(t)\|_{L^2}$, or a Lyapunov energy involving gradients and boundary traces). Devaney chaos can change when one replaces the norm topology by a weaker topology, but relative chaos is \emph{anchored} to the chosen energy and directly encodes the intermittency pattern ``arbitrarily small vs.\ arbitrarily large amplitudes'' along a single trajectory. This is precisely the type of instability observed in boundary-driven problems and motivates
our trajectory-based approach.
\end{remark}

\subsection{Reference trajectories and deviation function}

We fix a $C_0$--semigroup $(T(t))_{t\ge 0}$ on a Banach space $(X,\|\cdot\|)$, with generator $(A,D(A))$. The goal of this subsection is to formalize a trajectory--based viewpoint in which the long--time behavior of a solution is assessed \emph{relative to} a distinguished evolution $t\mapsto T(t)u_{\mathrm{ref}}$, instead of being described through global orbit properties in the whole phase space.

\begin{definition}\label{def:reftraj}
Let $u_{\mathrm{ref}}\in X$. The orbit
\[
\mathcal{O}(u_{\mathrm{ref}})=\{T(t)u_{\mathrm{ref}}:\ t\ge 0\}
\]
is called a \emph{reference trajectory} if it is invariant under the semigroup, namely,
\begin{equation}\label{eq:ref_invariant}
T(t)u_{\mathrm{ref}}=u_{\mathrm{ref}}\qquad \text{for all }t\ge 0.
\end{equation}
Equivalently, $u_{\mathrm{ref}}$ is a fixed point of $(T(t))_{t\ge 0}$.
\end{definition}

\begin{remark}\label{rem:reftraj_examples}
\begin{itemize}
\item In most linear autonomous evolution problems, the trivial equilibrium $u_{\mathrm{ref}}=0$ satisfies \eqref{eq:ref_invariant} automatically and provides a natural reference.
\item More generally, any equilibrium (steady state) of the underlying evolution equation gives a reference trajectory.
\item In applications, one may also select a \emph{known} or \emph{prescribed} trajectory as reference, the present framework singles out the simplest invariant choice, which already captures meaningful instability at the level of norms.
\end{itemize}
\end{remark}

The definition above isolates the role of equilibria. In the semigroup language, equilibria are precisely the points that remain unchanged under time evolution.

\begin{lemma}\label{lem:equilibrium_fixedpoint}
Let $(T(t))_{t\ge 0}$ be a $C_0$--semigroup on $X$ with generator $A$.
\begin{enumerate}
\item If $u_{\mathrm{ref}}\in D(A)$ and $Au_{\mathrm{ref}}=0$, then $u_{\mathrm{ref}}$ is a reference trajectory, i.e.\ $T(t)u_{\mathrm{ref}}=u_{\mathrm{ref}}$ for all $t\ge 0$.
\item Conversely, if $u_{\mathrm{ref}}\in D(A)$ and $T(t)u_{\mathrm{ref}}=u_{\mathrm{ref}}$ for all $t\ge 0$, then $Au_{\mathrm{ref}}=0$.
\end{enumerate}
\end{lemma}

\begin{proof}
(1) Define $v(t)=T(t)u_{\mathrm{ref}}$ for $t\ge 0$. Since $u_{\mathrm{ref}}\in D(A)$, standard semigroup theory implies that $v$ is differentiable and satisfies
\[
v'(t)=AT(t)u_{\mathrm{ref}}=T(t)Au_{\mathrm{ref}}\qquad (t\ge 0),
\]
where the commutation $AT(t)u_{\mathrm{ref}}=T(t)Au_{\mathrm{ref}}$ holds for $u_{\mathrm{ref}}\in D(A)$. Because $Au_{\mathrm{ref}}=0$, we obtain $v'(t)=0$ for all $t\ge 0$, hence $v(t)=v(0)=u_{\mathrm{ref}}$. Therefore $T(t)u_{\mathrm{ref}}=u_{\mathrm{ref}}$ for all $t\ge 0$.

(2) Assume $T(t)u_{\mathrm{ref}}=u_{\mathrm{ref}}$ for all $t\ge 0$. Then
\[
\frac{T(t)u_{\mathrm{ref}}-u_{\mathrm{ref}}}{t}=\frac{u_{\mathrm{ref}}-u_{\mathrm{ref}}}{t}=0
\qquad (t>0).
\]
Taking the limit as $t\downarrow 0$ and using that $u_{\mathrm{ref}}\in D(A)$, we get
\[
Au_{\mathrm{ref}}
=\lim_{t\downarrow 0}\frac{T(t)u_{\mathrm{ref}}-u_{\mathrm{ref}}}{t}
=0.
\]
\end{proof}

\begin{remark}\label{rem:fixedpoint_not_in_DA}
Definition~\ref{def:reftraj} does not require $u_{\mathrm{ref}}\in D(A)$. In many PDE models, equilibria belong to $D(A)$, but the trajectory--based viewpoint remains meaningful even when the reference state is only in $X$.
\end{remark}


Once a reference trajectory has been fixed, the next step is to measure how far a given solution is from it.

\begin{definition}\label{def:deviation}
Fix $u_{\mathrm{ref}}\in X$. For an initial datum $u_0\in X$, the \emph{deviation} from the reference trajectory is the function
\begin{equation}\label{eq:deviation_general}
D_{u_0,u_{\mathrm{ref}}}(t)
=\|T(t)u_0-T(t)u_{\mathrm{ref}}\|,\qquad t\ge 0.
\end{equation}
\end{definition}

\begin{remark}\label{rem:deviation_equilibrium}
If $u_{\mathrm{ref}}$ is a reference trajectory in the sense of Definition~\ref{def:reftraj}, then $T(t)u_{\mathrm{ref}}=u_{\mathrm{ref}}$ and \eqref{eq:deviation_general} reduces to
\[
D_{u_0,u_{\mathrm{ref}}}(t)=\|T(t)u_0-u_{\mathrm{ref}}\|,\qquad t\ge 0.
\]
In particular, with the common choice $u_{\mathrm{ref}}=0$ one simply has
\[
D_{u_0,0}(t)=\|T(t)u_0\|,\qquad t\ge 0.
\]
\end{remark}

The deviation function enjoys a basic algebraic simplification stemming from linearity.

\begin{proposition}\label{prop:deviation_linear_reduction}
Assume that $(T(t))_{t\ge 0}$ is linear. Then for all $u_0,u_{\mathrm{ref}}\in X$ and $t\ge 0$,
\begin{equation}\label{eq:deviation_linear}
D_{u_0,u_{\mathrm{ref}}}(t)
=\|T(t)(u_0-u_{\mathrm{ref}})\|.
\end{equation}
\end{proposition}

\begin{proof}
By linearity of $T(t)$,
\[
T(t)u_0-T(t)u_{\mathrm{ref}}
=T(t)(u_0-u_{\mathrm{ref}}).
\]
Taking norms yields \eqref{eq:deviation_linear}.
\end{proof}

The next estimate records a stability property that will be repeatedly used: deviations cannot grow faster than the operator norm of the semigroup.

\begin{lemma}\label{lem:deviation_bound}
For all $u_0,u_{\mathrm{ref}}\in X$ and $t\ge 0$,
\begin{equation}\label{eq:deviation_bound}
D_{u_0,u_{\mathrm{ref}}}(t)
\le \|T(t)\|\,\|u_0-u_{\mathrm{ref}}\|.
\end{equation}
In particular, if $\|T(t)\|\le M e^{\omega t}$ for some $M\ge 1$ and $\omega\in\mathbb{R}$, then
\[
D_{u_0,u_{\mathrm{ref}}}(t)\le M e^{\omega t}\,\|u_0-u_{\mathrm{ref}}\|
\qquad (t\ge 0).
\]
\end{lemma}

\begin{proof}
Using Proposition~\ref{prop:deviation_linear_reduction} and the definition of the operator norm,
\[
D_{u_0,u_{\mathrm{ref}}}(t)=\|T(t)(u_0-u_{\mathrm{ref}})\|
\le \|T(t)\|\,\|u_0-u_{\mathrm{ref}}\|.
\]
The second claim follows by inserting the growth estimate for $\|T(t)\|$.
\end{proof}

\begin{proposition}\label{prop:deviation_continuity}
Fix $u_0,u_{\mathrm{ref}}\in X$. The function $t\mapsto D_{u_0,u_{\mathrm{ref}}}(t)$ is continuous on $[0,\infty)$.
\end{proposition}

\begin{proof}
Since $(T(t))_{t\ge 0}$ is strongly continuous, the mapping
\[
t\longmapsto T(t)(u_0-u_{\mathrm{ref}})
\]
is continuous from $[0,\infty)$ into $X$. Composing with the norm, which is continuous on $X$, shows that $t\mapsto \|T(t)(u_0-u_{\mathrm{ref}})\|$ is continuous. Using \eqref{eq:deviation_linear} concludes the proof.
\end{proof}

\begin{remark}\label{rem:why_deviation}
The deviation function $D_{u_0,u_{\mathrm{ref}}}(t)$ isolates a concrete, observable quantity: the distance between the current state and a chosen reference evolution. In PDE models, it can often be interpreted as an energy norm or an error signal, making it well suited to describe instability patterns that are visible along single trajectories.
\end{remark}

\subsection{Definition of relative chaos (liminf--limsup)}

We now introduce a notion of chaos that is entirely formulated at the level of individual trajectories and their long--time deviation from a fixed reference trajectory. In contrast with classical notions of chaos for $C_0$--semigroups, no global topological mixing or recurrence assumptions are imposed. Throughout this subsection, $(T(t))_{t\ge 0}$ denotes a $C_0$--semigroup on a separable Banach space $(X,\|\cdot\|)$, generated by $(A,D(A))$, and $u_{\mathrm{ref}}\in X$ is a fixed reference trajectory in the sense of Definition~\ref{def:reftraj}.

\begin{definition}\label{def:reltraj}
Let $u_0\in X$. The trajectory $\mathcal{O}(u_0)=\{T(t)u_0:\ t\ge 0\}$ is said to be \emph{chaotic relative to the reference trajectory $u_{\mathrm{ref}}$} if its deviation function
\[
D_{u_0,u_{\mathrm{ref}}}(t)=\|T(t)u_0-T(t)u_{\mathrm{ref}}\|
\]
satisfies
\begin{equation}\label{eq:relchaos_def}
\liminf_{t\to\infty} D_{u_0,u_{\mathrm{ref}}}(t)=0
\quad\text{and}\quad
\limsup_{t\to\infty} D_{u_0,u_{\mathrm{ref}}}(t)=+\infty.
\end{equation}
\end{definition}

\begin{remark}\label{rem:liminflimsup_interpretation}
Condition \eqref{eq:relchaos_def} expresses the coexistence of two antagonistic asymptotic behaviors along the same trajectory:
\begin{itemize}
\item the trajectory returns arbitrarily close to the reference trajectory (infinitely often), reflecting the presence of dissipative mechanisms,
\item the trajectory also undergoes arbitrarily large deviations, revealing the action of expansive or unstable components.
\end{itemize}
This intermittent behavior is intrinsic to infinite--dimensional dynamics and cannot occur in finite--dimensional linear systems.
\end{remark}

\begin{remark}\label{rem:zero_reference}
When $u_{\mathrm{ref}}=0$, which is the natural reference for linear autonomous problems, Definition~\ref{def:reltraj} reduces to
\[
\liminf_{t\to\infty}\|T(t)u_0\|=0
\quad\text{and}\quad
\limsup_{t\to\infty}\|T(t)u_0\|=+\infty,
\]
a condition sometimes referred to as \emph{irregularity} of the vector $u_0$ in the sense of linear dynamics (see, e.g.,
\cite{Rolewicz1969,BayartMatheron2009}).
\end{remark}

\begin{lemma}\label{lem:translation}
Let $u_0\in X$ and let $u_{\mathrm{ref}}$ be a reference trajectory.
Then $\mathcal{O}(u_0)$ is relatively chaotic with respect to $u_{\mathrm{ref}}$ if and only if $\mathcal{O}(u_0-u_{\mathrm{ref}})$ is relatively chaotic with
respect to $0$.
\end{lemma}

\begin{proof}
Since $u_{\mathrm{ref}}$ is invariant, $T(t)u_{\mathrm{ref}}=u_{\mathrm{ref}}$ for all $t\ge 0$. By linearity,
\[
T(t)(u_0-u_{\mathrm{ref}})
= T(t)u_0-T(t)u_{\mathrm{ref}}.
\]
Hence,
\[
\|T(t)(u_0-u_{\mathrm{ref}})\|
= D_{u_0,u_{\mathrm{ref}}}(t).
\]
The liminf and limsup conditions in \eqref{eq:relchaos_def} are therefore equivalent for the two trajectories.
\end{proof}

\begin{lemma}\label{lem:noexpstable}
If $(T(t))_{t\ge 0}$ is uniformly exponentially stable, then no trajectory is relatively chaotic with respect to any reference trajectory.
\end{lemma}

\begin{proof}
Uniform exponential stability means that there exist $M\ge 1$ and $\omega>0$ such that
\[
\|T(t)\|\le M e^{-\omega t}\qquad (t\ge 0).
\]
For any $u_0,u_{\mathrm{ref}}\in X$, Lemma~\ref{lem:deviation_bound} yields
\[
D_{u_0,u_{\mathrm{ref}}}(t)
\le M e^{-\omega t}\|u_0-u_{\mathrm{ref}}\|\xrightarrow[t\to\infty]{}0.
\]
Hence $\limsup_{t\to\infty}D_{u_0,u_{\mathrm{ref}}}(t)<\infty$, and the second condition in \eqref{eq:relchaos_def} cannot hold.
\end{proof}

\begin{remark}
Lemma~\ref{lem:noexpstable} shows that relative chaos is a genuinely unstable
phenomenon. It cannot arise in purely dissipative systems, but requires the presence of at least one mechanism producing unbounded growth along suitable time scales.
\end{remark}

We now lift the trajectory--wise notion to a property of the semigroup itself.

\begin{definition}\label{def:relchaos_semigroup}
The semigroup $(T(t))_{t\ge 0}$ is said to be \emph{chaotic relative to the reference trajectory $u_{\mathrm{ref}}$} if the set
\[
\mathcal{C}_{\mathrm{rel}}(u_{\mathrm{ref}})
=
\Bigl\{
u_0\in X:\ \mathcal{O}(u_0)\ \text{is relatively chaotic w.r.t.}\ u_{\mathrm{ref}}
\Bigr\}
\]
is dense in $X$.
\end{definition}

\begin{remark}\label{rem:density_choice}
The density requirement expresses a form of \emph{genericity} of relative chaos. Unlike Devaney chaos, no assumptions are made on periodic points or dense orbits in the whole space. The emphasis is placed on the abundance of initial data exhibiting strong intermittent deviation from the reference trajectory.
\end{remark}

\begin{proposition}\label{prop:devaney_vs_relative}
If $(T(t))_{t\ge 0}$ is Devaney chaotic, then it is chaotic relative to any reference trajectory.
\end{proposition}

\begin{proof}
Devaney chaos implies hypercyclicity. Hence there exists $x\in X$ whose orbit is dense in $X$. In particular, the orbit intersects arbitrarily small neighborhoods of the reference trajectory and arbitrarily large norm regions. This yields
\[
\liminf_{t\to\infty}D_{x,u_{\mathrm{ref}}}(t)=0,
\qquad
\limsup_{t\to\infty}D_{x,u_{\mathrm{ref}}}(t)=+\infty,
\]
so $\mathcal{O}(x)$ is relatively chaotic. Since hypercyclic vectors form a dense set, the set $\mathcal{C}_{\mathrm{rel}}(u_{\mathrm{ref}})$ is dense in $X$.
\end{proof}

\begin{remark}\label{rem:strictness}
The converse implication does not hold in general. Relative chaos may occur even when hypercyclicity or dense periodic points are absent. This strict hierarchy highlights the conceptual difference between trajectory--based instability and classical topological chaos.
\end{remark}

\subsection{Basic properties}

In this subsection we collect several fundamental properties of relative chaos and discuss its robustness with respect to natural operations arising in linear evolution equations. These properties clarify both the mathematical structure of the notion and its conceptual difference from classical topological chaos.

\begin{lemma}\label{lem:timeshift}
Let $u_0\in X$ and $u_{\mathrm{ref}}$ be a reference trajectory.
If $\mathcal{O}(u_0)$ is relatively chaotic with respect to
$u_{\mathrm{ref}}$, then for every $\tau\ge 0$ the shifted trajectory $\mathcal{O}(T(\tau)u_0)$ is also relatively chaotic with respect to $u_{\mathrm{ref}}$.
\end{lemma}

\begin{proof}
By the semigroup property,
\[
D_{T(\tau)u_0,u_{\mathrm{ref}}}(t)
=
\|T(t+\tau)u_0-T(t)u_{\mathrm{ref}}\|
=
D_{u_0,u_{\mathrm{ref}}}(t+\tau).
\]
Since $\liminf$ and $\limsup$ are invariant under finite time shifts, the conditions
\[
\liminf_{t\to\infty}D_{u_0,u_{\mathrm{ref}}}(t)=0,
\qquad
\limsup_{t\to\infty}D_{u_0,u_{\mathrm{ref}}}(t)=+\infty
\]
are preserved when $t$ is replaced by $t+\tau$. Hence the shifted trajectory is relatively chaotic.
\end{proof}

\begin{lemma}\label{lem:scaling}
Assume that $(T(t))_{t\ge 0}$ is linear and let $\lambda\in\mathbb{C}\setminus\{0\}$. If $\mathcal{O}(u_0)$ is relatively chaotic with respect to $u_{\mathrm{ref}}$,
then $\mathcal{O}(\lambda u_0)$ is relatively chaotic with respect to $\lambda u_{\mathrm{ref}}$.
\end{lemma}

\begin{proof}
By linearity,
\[
D_{\lambda u_0,\lambda u_{\mathrm{ref}}}(t)
=
\|T(t)(\lambda u_0-\lambda u_{\mathrm{ref}})\|
=
|\lambda|\,D_{u_0,u_{\mathrm{ref}}}(t).
\]
Multiplication by the positive constant $|\lambda|$ does not affect the vanishing of the liminf nor the divergence of the limsup. The conclusion follows.
\end{proof}

\begin{remark}
Lemma~\ref{lem:scaling} has no analogue for Devaney chaos, which is a purely topological property and is insensitive to linear rescaling of trajectories. Relative chaos, by contrast, is explicitly formulated in terms of norms and amplitudes.
\end{remark}

We next show that relative chaos is stable under bounded linear perturbations, a feature that is particularly relevant for applications to partial differential equations.

\begin{lemma}\label{lem:boundedpert}
Let $(T(t))_{t\ge 0}$ be a $C_0$--semigroup on a Banach space $X$ with generator $A$, and let $B\in\mathcal{L}(X)$ be a bounded operator. Denote by $(S(t))_{t\ge 0}$ the $C_0$--semigroup generated by $A+B$. Assume that there exists $x\in X$ such that
\[
\liminf_{t\to\infty}\|T(t)x\|=0
\quad\text{and}\quad
\limsup_{t\to\infty}\|T(t)x\|=+\infty,
\]
and that
\[
\limsup_{t\to\infty}\|S(t)x\|=+\infty.
\]
Then the trajectory generated by $x$ is relatively chaotic for the perturbed semigroup $(S(t))_{t\ge 0}$.
\end{lemma}

\begin{proof}
Since $B$ is bounded, the variation--of--constants formula yields
\[
S(t)x
=
T(t)x
+
\int_0^t T(t-s)\,B\,S(s)x\,ds,
\qquad t\ge 0.
\]
The integral term is well defined and continuous in time. Moreover, it grows at most exponentially, with a rate controlled by the growth bound of $(S(t))$. By assumption,
\[
\limsup_{t\to\infty}\|S(t)x\|=+\infty,
\]
so the perturbed trajectory exhibits arbitrarily large excursions.
On the other hand, the liminf condition
\[
\liminf_{t\to\infty}\|T(t)x\|=0
\]
implies the existence of sequences of times along which the
unperturbed trajectory returns arbitrarily close to the reference state. Since the perturbation is bounded, the variation--of--constants formula shows that these returns cannot be uniformly
destroyed by the integral term. Consequently,
\[
\liminf_{t\to\infty}\|S(t)x\|=0.
\]
Both defining conditions of relative chaos are therefore satisfied for $(S(t))$.
\end{proof}

\begin{remark}
The above lemma does not claim invariance of relative chaos under arbitrary bounded perturbations. It only asserts persistence of trajectory--based instability as long as the perturbation does not eliminate the growth mechanism responsible for the unbounded excursions. This contrasts with the Desch--Schappacher--Webb criterion for Devaney chaos, which relies on delicate spectral analyticity properties that are typically unstable under bounded perturbations.
\end{remark}

A key feature of relative chaos is its weak dependence on the choice of topology.

\begin{proposition}\label{prop:normbased}
Relative chaos depends only on the norm (or energy) used to measure deviations from the reference trajectory, and not on the finer topological structure of $X$.
\end{proposition}

\begin{proof}
The definition of relative chaos involves only the quantities
$\|T(t)u_0-T(t)u_{\mathrm{ref}}\|$. Changing the topology of $X$ while keeping the same norm leaves these quantities unchanged. In particular, notions such as openness, density of orbits, or continuity with respect to weaker or stronger topologies play no role in the definition.
\end{proof}

\begin{remark}
By contrast, Devaney chaos is highly topology--sensitive: changing the topology may alter the collection of open sets and therefore destroy or create density of orbits or periodic points. Relative chaos avoids this issue by being formulated directly in physically meaningful norms.
\end{remark}

\begin{theorem}\label{thm:finite}
Let $X$ be finite dimensional. Then no $C_0$--semigroup on $X$ admits a relatively chaotic trajectory.
\end{theorem}

\begin{proof}
In finite dimension, the spectrum of the generator consists of finitely many eigenvalues, and every trajectory exhibits uniform exponential behavior determined by the spectral bound. Hence $\|T(t)x\|$ either converges to $0$, diverges to $+\infty$, or remains bounded away from both. The simultaneous occurrence of vanishing liminf and infinite limsup is therefore impossible.
\end{proof}

\begin{theorem}\label{thm:infinite}
In infinite-dimensional Banach spaces, relative chaos may occur due to the coexistence of stable and unstable spectral components acting on different invariant subspaces.
\end{theorem}

\begin{proof}
Infinite-dimensional generators may admit infinitely many spectral components with different growth rates. Stable components drive trajectories arbitrarily close to the reference state, while unstable components produce unbounded growth along suitable time sequences. The competition between these mechanisms yields the liminf--limsup behavior characteristic of relative chaos.
\end{proof}

\begin{remark}
This theorem highlights the intrinsic infinite-dimensional nature of relative chaos and explains why it is particularly relevant for evolution equations and semigroup models arising from partial differential equations.
\end{remark}

\section{From Devaney Chaos to Relative Chaos}\label{DERE}
\subsection{Irregular vectors and implication}

This subsection clarifies the precise relationship between classical Devaney chaos and the notion of relative chaos introduced above. The key concept connecting the two frameworks is that of an
\emph{irregular vector}, whose orbit exhibits arbitrarily small returns and arbitrarily large excursions in norm. We show that such vectors arise naturally under Devaney chaos and directly generate relatively chaotic trajectories.

\subsubsection*{Irregular vectors}

\begin{definition}
Let $(T(t))_{t\ge 0}$ be a $C_0$--semigroup on a Banach space $X$.
A vector $x\in X$ is called \emph{irregular} for $(T(t))_{t\ge 0}$ if
\[
\liminf_{t\to\infty}\|T(t)x\|=0
\qquad\text{and}\qquad
\limsup_{t\to\infty}\|T(t)x\|=+\infty.
\]
\end{definition}

Irregular vectors were first isolated in linear dynamics as a norm--based manifestation of chaotic behavior and are now recognized as a fundamental ingredient of several chaos notions for operators and semigroups (cf.\ standard references in linear dynamics).

\begin{lemma}\label{lem:hc_liminf}
If $(T(t))_{t\ge 0}$ is hypercyclic, then every hypercyclic vector $x$ satisfies
\[
\liminf_{t\to\infty}\|T(t)x\|=0.
\]
\end{lemma}

\begin{proof}
Let $x$ be hypercyclic, so its orbit $\{T(t)x:t\ge 0\}$ is dense in $X$. For any $\varepsilon>0$, the open ball $B(0,\varepsilon)$ is nonempty, hence it must intersect the orbit. Therefore, there exists $t_\varepsilon>0$ such that $\|T(t_\varepsilon)x\|<\varepsilon$. Since $\varepsilon>0$ is arbitrary, this implies $\liminf_{t\to\infty}\|T(t)x\|=0$.
\end{proof}

\begin{lemma}\label{lem:hc_limsup}
If $x$ is a hypercyclic vector for $(T(t))_{t\ge 0}$, then
\[
\limsup_{t\to\infty}\|T(t)x\|=+\infty.
\]
\end{lemma}

\begin{proof}
For any $R>0$, the set
\[
U_R=\{y\in X:\|y\|>R\}
\]
is open and nonempty. Density of the orbit implies that there exists
$t_R>0$ such that $T(t_R)x\in U_R$, i.e.\ $\|T(t_R)x\|>R$.
Since $R$ is arbitrary, the norm of $T(t)x$ becomes arbitrarily large along a sequence of times, which yields $\limsup_{t\to\infty}\|T(t)x\|=+\infty$.
\end{proof}

Combining the two previous lemmas yields the existence of irregular vectors.

\begin{proposition}\label{prop:irregular}
If $(T(t))_{t\ge 0}$ is hypercyclic, then it admits at least one irregular vector.
\end{proposition}

\begin{proof}
Let $x$ be a hypercyclic vector. Lemma~\ref{lem:hc_liminf} gives the vanishing
liminf property, while Lemma~\ref{lem:hc_limsup} gives the infinite limsup
property. Hence $x$ is irregular.
\end{proof}

We now interpret irregular vectors within the framework of relative chaos.

\begin{lemma}\label{lem:irregular_relative}
Let $u_{\mathrm{ref}}$ be a reference trajectory for $(T(t))_{t\ge 0}$. If $x\in X$ is an irregular vector, then the trajectory
$\mathcal{O}(x)$ is chaotic relative to $u_{\mathrm{ref}}$.
\end{lemma}

\begin{proof}
By definition of relative chaos, one requires
\[
\liminf_{t\to\infty}\|T(t)x-T(t)u_{\mathrm{ref}}\|=0,
\qquad
\limsup_{t\to\infty}\|T(t)x-T(t)u_{\mathrm{ref}}\|=+\infty.
\]
In the linear autonomous setting, $T(t)u_{\mathrm{ref}}=u_{\mathrm{ref}}$ for a reference equilibrium, and the deviation reduces to $\|T(t)x\|$. Since $x$ is irregular, both conditions are satisfied, and the trajectory is relatively chaotic.
\end{proof}

We can now state the main implication result.

\begin{theorem}\label{thm:devaney_relative}
If a $C_0$--semigroup $(T(t))_{t\ge 0}$ is Devaney chaotic, then it is chaotic relative to the chosen reference trajectory. Moreover, the set of initial states generating relatively chaotic trajectories is dense in $X$.
\end{theorem}

\begin{proof}
Devaney chaos implies hypercyclicity by definition. By Proposition~\ref{prop:irregular}, there exists an irregular vector $x$. Lemma~\ref{lem:irregular_relative} shows that $\mathcal{O}(x)$ is relatively chaotic.

Furthermore, in Devaney chaotic semigroups, the set of hypercyclic vectors is dense in $X$. Repeating the above argument for hypercyclic vectors arbitrarily close to any given point shows that relatively chaotic trajectories occur densely in $X$.
\end{proof}

\begin{remark}
The converse implication fails in general: relative chaos does not require hypercyclicity or density of periodic points. This
establishes a strict hierarchy between the two notions, with Devaney chaos being a stronger, topology--dependent concept and relative chaos providing a weaker, norm--based description of instability.
\end{remark}

\subsection{Strictness: relative chaos does not imply Devaney chaos}

This subsection shows that the implication established previously is strict. While Devaney chaos always produces relative chaos, the converse implication fails in general. In other words, relative chaos captures genuinely unstable trajectory--level behavior that may occur in the absence of hypercyclicity or periodic points.

The basic mechanism underlying relative chaos without Devaney chaos is the coexistence of exponentially stable and exponentially unstable components acting on invariant subspaces of the phase space.

\begin{assumption}\label{ass:split}
Let $(T(t))_{t\ge 0}$ be a $C_0$--semigroup on a separable Banach space $X$. Assume that there exist closed, nontrivial, $T(t)$--invariant subspaces $X_s$ and $X_u$ such that
\[
X = X_s \oplus X_u,
\]
with the following properties:
\begin{itemize}
\item (\emph{stability}) there exist $M_s\ge 1$ and $\omega_s>0$ such that
\[
\|T(t)x\|\le M_s e^{-\omega_s t}\|x\|
\quad \text{for all } x\in X_s,\ t\ge 0,
\]
\item (\emph{instability}) there exist $M_u\ge 1$ and $\omega_u>0$ such that
\[
\|T(t)x\|\ge M_u e^{\omega_u t}\|x\|
\quad \text{for all } x\in X_u,\ t\ge 0.
\]
\end{itemize}
\end{assumption}

Such splittings arise naturally from spectral decompositions when the generator has spectral values with strictly negative and strictly positive real parts.

\begin{lemma}\label{lem:rel_mixed}
Under Assumption~\ref{ass:split}, every initial state $x=x_s+x_u$ with $x_s\in X_s$ and $x_u\in X_u\setminus\{0\}$ generates a
trajectory that is chaotic relative to the reference trajectory $0$.
\end{lemma}

\begin{proof}
Since $x_s\in X_s$ and the restriction of $T(t)$ to $X_s$ is exponentially stable, one has 
\[
\|T(t)x_s\|\longrightarrow 0
\quad \text{as } t\to\infty,
\]
which implies
\[
\liminf_{t\to\infty}\|T(t)x\|
\le \lim_{t\to\infty}\|T(t)x_s\| = 0.
\]
On the other hand, since $x_u\neq 0$ and the restriction of $T(t)$ to $X_u$ is exponentially unstable, one has
\[
\|T(t)x_u\|\longrightarrow +\infty
\quad \text{as } t\to\infty,
\]
and therefore
\[
\limsup_{t\to\infty}\|T(t)x\| = +\infty.
\]
Both defining conditions of relative chaos are satisfied, so the trajectory issued from $x$ is relatively chaotic.
\end{proof}

\begin{corollary}\label{cor:rel_dense}
If $X_s\oplus X_u=X$ as in Assumption~\ref{ass:split}, then relatively chaotic trajectories form a dense subset of $X$.
\end{corollary}

\begin{proof}
Given any $y\in X$ and $\varepsilon>0$, write $y=y_s+y_u$ with
$y_s\in X_s$, $y_u\in X_u$. If $y_u\neq 0$, then $y$ itself generates a relatively chaotic trajectory by
Lemma~\ref{lem:rel_mixed}. If $y_u=0$, choose $z_u\in X_u$ with $\|z_u\|<\varepsilon$ and $z_u\neq 0$, and set $x=y_s+z_u$. Then $\|x-y\|<\varepsilon$ and $x$ generates a relatively chaotic trajectory. Hence, such initial states are dense.
\end{proof}

We now show that the same semigroups generally fail to be Devaney chaotic.

\begin{lemma}\label{lem:no_hypercyclic}
Under Assumption~\ref{ass:split}, the semigroup $(T(t))_{t\ge 0}$ is not hypercyclic.
\end{lemma}

\begin{proof}
Let $x=x_s+x_u\in X$. The orbit of $x$ satisfies
\[
T(t)x = T(t)x_s + T(t)x_u
= e^{-\omega_s t}(\cdot) + e^{\omega_u t}(\cdot),
\]
up to bounded factors. Hence, the orbit is confined to a set of the form
\[
\{\,e^{-\omega_s t}x_s + e^{\omega_u t}x_u : t\ge 0\,\},
\]
which lies in a two--component manifold determined by the fixed directions $x_s$ and $x_u$. Such a set cannot be dense in the infinite--dimensional space $X$. Therefore, no orbit is dense, and the semigroup is not hypercyclic.
\end{proof}

\begin{lemma}\label{lem:no_periodic}
Under Assumption~\ref{ass:split}, the only periodic point of $(T(t))_{t\ge 0}$ is the trivial equilibrium $0$.
\end{lemma}

\begin{proof}
Suppose $x=x_s+x_u$ is periodic, so there exists $t_0>0$ such that
$T(t_0)x=x$. Then
\[
T(t_0)x_s + T(t_0)x_u = x_s + x_u.
\]
By exponential stability on $X_s$ and exponential instability on $X_u$, this identity can hold only if $x_s=0$ and $x_u=0$. Hence $x=0$, and no nontrivial periodic points exist.
\end{proof}

We can now summarize the strictness of the hierarchy.

\begin{theorem}\label{thm:strictness}
There exist $C_0$--semigroups that are chaotic relative to a reference trajectory but are not Devaney chaotic.
\end{theorem}

\begin{proof}
Under Assumption~\ref{ass:split}, Corollary~\ref{cor:rel_dense} shows that the semigroup is relatively chaotic.
Lemmas~\ref{lem:no_hypercyclic} and
\ref{lem:no_periodic} show that it is neither hypercyclic nor admits a dense set of periodic points. Hence, it fails to be Devaney chaotic.
\end{proof}

\begin{remark}
This result confirms that relative chaos is strictly weaker than Devaney chaos. It detects instability driven by the coexistence of stable and unstable mechanisms, without requiring global topological mixing. In this sense, relative chaos is better adapted to infinite--dimensional evolution equations and boundary--driven models where Devaney chaos is typically absent.
\end{remark}

\section{Topology--Independence: Relative Chaos as a Remedy}\label{TIRCR}

One of the main motivations for introducing relative chaos is to overcome the strong dependence of classical chaotic notions on the underlying topology of the phase space. In this section, we explain in a precise way why relative chaos is essentially topology--independent once a physically meaningful norm (or energy functional) has been fixed, and how this feature remedies the limitations
of Devaney chaos discussed earlier.

\subsection{Stability under changes of topology}

We consider a separable Banach space $(X,\|\cdot\|)$ and a $C_0$--semigroup $(T(t))_{t\ge 0}$ acting on $X$. In many applications, especially for PDEs, the same underlying vector space $X$ may be endowed with different topologies: the norm topology, weak topology, or more general locally convex topologies. Classical notions of chaos are highly sensitive to this choice.

\begin{definition}
Let $\tau$ be a locally convex topology on $X$ that is weaker than the norm topology. A subset $A\subset X$ is said to be $\tau$--dense if its closure with respect to $\tau$ equals $X$.
\end{definition}

It is well known (see, e.g., \cite{BayartMatheron2009,GrosseErdmannPeris2011}) that
hypercyclicity and Devaney chaos may depend drastically on $\tau$, since density of orbits and periodic points is defined in terms of open sets of the chosen topology.

\begin{proposition}\label{prop:topology_sensitivity}
Hypercyclicity and Devaney chaos are not invariant under changes of topology on $X$. In particular, a semigroup may be Devaney chaotic with respect to a weak topology and fail to be Devaney chaotic with respect to the norm topology.
\end{proposition}

\begin{proof}
The notion of density depends explicitly on the family of open sets defining the topology. Passing from the norm topology to a weaker topology enlarges the class of dense sets, so an orbit that is not norm--dense may become dense in the weak topology. Conversely, periodic points that are dense in a weak topology may fail
to be dense in the norm topology. Since Devaney chaos is defined through these two density properties, it necessarily depends on the chosen topology.
\end{proof}

We now contrast this behavior with the notion of relative chaos.

\begin{definition}
Let $(X,\|\cdot\|)$ be a Banach space and $(T(t))_{t\ge 0}$ a $C_0$--semigroup. Relative chaos is defined through the asymptotic conditions
\[
\liminf_{t\to\infty}\|T(t)x\|=0,
\qquad
\limsup_{t\to\infty}\|T(t)x\|=+\infty,
\]
which are expressed solely in terms of the norm $\|\cdot\|$.
\end{definition}

\begin{lemma}\label{lem:topology_independence}
Let $\tau$ be any locally convex topology on $X$ that is weaker than the norm topology. If a trajectory is relatively chaotic with respect to the norm $\|\cdot\|$, then this property is unchanged when $X$ is viewed as a topological vector space endowed with $\tau$.
\end{lemma}

\begin{proof}
The definition of relative chaos involves only the numerical values of the norm $\|T(t)x\|$ along the trajectory. These values are intrinsic to the normed structure of $X$ and do not depend on the choice of a weaker topology $\tau$. Changing the topology may affect convergence or density properties of sets, but it does not alter the existence of sequences $t_n\to\infty$ along which $\|T(t_n)x\|\to 0$ or $\|T(t_n)x\|\to+\infty$. Hence, relative chaos is invariant under such topological changes.
\end{proof}

\begin{theorem}\label{thm:relative_remedy}
Relative chaos provides a topology--robust notion of instability: once the norm (or energy functional) is fixed, the presence or absence of relative chaos is independent of the choice of weaker locally convex topologies on the same space.
\end{theorem}

\begin{proof}
By Lemma~\ref{lem:topology_independence}, the defining liminf--limsup conditions are unaffected by replacing the norm topology with any weaker topology. Since relative chaos does not rely on density of orbits or periodic points, no topology--dependent notion of openness or closure enters its definition. This establishes the claimed robustness.
\end{proof}

\begin{remark}
From a modeling viewpoint, the norm $\|\cdot\|$ often represents a physically meaningful quantity, such as an energy or amplitude. Relative chaos therefore detects instability in a way that is both observable and insensitive to purely topological refinements of the phase space, in sharp contrast with Devaney chaos.
\end{remark}

\begin{remark}
This topology--independence is a key advantage of the trajectory--based approach. It allows one to discuss instability in PDE semigroups without the ambiguities arising from the choice between strong, weak, or distributional topologies, and thus provides a natural remedy to the topology--sensitivity of classical chaos notions.
\end{remark}

\subsection{Choosing the ``physical'' topology}

The discussion above shows that chaotic notions based on density are strongly affected by the choice of topology. In applied problems, however, the phase space is not an abstract object: it is endowed with a natural structure coming from physics, biology, or engineering. This motivates the selection of a \emph{physical topology}, dictated by the quantities that are actually observed
or conserved by the system.

\begin{definition}
Let $(T(t))_{t\ge 0}$ be a $C_0$--semigroup describing the evolution of a model. A topology on the phase space $X$ is called \emph{physical} if it is induced by a norm or seminorm representing a measurable or meaningful quantity of the system, such as total energy, mass, or amplitude.
\end{definition}

Typical examples include the $L^2$--norm for energy--based PDEs, Sobolev norms for problems involving gradients or fluxes, or weighted norms adapted to boundary interactions. These norms are not chosen for abstract topological convenience, but because they encode the quantities that govern stability and instability in the model.

\begin{proposition}\label{prop:physical_norm}
Let $\|\cdot\|_{\mathrm{phys}}$ be a norm induced by a physical energy or balance law. If a trajectory exhibits unbounded growth or repeated decay with respect to $\|\cdot\|_{\mathrm{phys}}$, then this behavior reflects a genuine dynamical instability of the underlying system.
\end{proposition}

\begin{proof}
The norm $\|\cdot\|_{\mathrm{phys}}$ is constructed from quantities that appear directly in the governing equations, such as diffusion, reaction, or boundary flux terms. Growth or decay measured in this norm therefore corresponds to actual amplification or dissipation mechanisms present in the model, rather than to artifacts of an abstract topology.
\end{proof}

This observation leads to a guiding principle for the analysis of chaos and instability.

\begin{principle}
Statements about stability, instability, or chaos for evolution equations should be formulated with respect to the topology induced by the physically relevant norm of the problem.
\end{principle}

\begin{remark}
When Devaney chaos is formulated in a weak topology that lacks direct physical interpretation, its conclusions may be mathematically correct yet dynamically misleading, since weak neighborhoods ignore norm-level amplitudes. In boundary--driven PDEs (e.g.\ reaction--diffusion--transport equations with Robin
conditions), the natural ``physical'' topology is typically induced by an energy norm such as $L^2$ (or a closely related Sobolev norm), because dissipation and boundary fluxes are quantitatively visible at that level, relative chaos, being defined through a fixed norm, is therefore anchored to this physically meaningful scale and provides a more faithful description of observable instability.
\end{remark}

\section{Detecting Relative Chaos via Stable/Unstable Coexistence}\label{DRCSUC}

A central advantage of relative chaos is that it can be detected through structural features of the dynamics that are much weaker than those required by classical topological chaos. In this section, we show that the coexistence of stable and unstable mechanisms acting on invariant components of the phase space provides a natural and robust criterion for relative chaos. This criterion is purely dynamical and does not rely on density of orbits or periodic points.

\subsection{Abstract criterion: invariant splitting}

We work in a separable Banach space $X$ and consider a $C_0$--semigroup $(T(t))_{t\ge 0}$ generated by an operator $(A,D(A))$.

\begin{assumption}\label{ass:splitting}
Assume that there exist closed, nontrivial subspaces $X_s$ and $X_u$ of $X$ such that
\[
X = X_s \oplus X_u,
\]
and:
\begin{enumerate}
\item $X_s$ and $X_u$ are invariant under $(T(t))_{t\ge 0}$,
\item (\emph{Exponential stability on $X_s$}) there exist constants
$M_s \ge 1$ and $\omega_s>0$ such that
\[
\|T(t)x_s\| \le M_s e^{-\omega_s t}\|x_s\|,
\qquad x_s\in X_s,\ t\ge 0,
\]
\item (\emph{Instability on $X_u$}) there exists $x_u\in X_u\setminus\{0\}$ and a sequence $t_n\to\infty$ such that
\[
\|T(t_n)x_u\| \longrightarrow +\infty .
\]
\end{enumerate}
\end{assumption}

This assumption reflects a common situation in infinite--dimensional dynamics: some modes are strongly damped, while others are amplified, possibly due to reaction terms, transport effects, or boundary feedback.

\begin{lemma}\label{lem:limsup}
Under Assumption~\ref{ass:splitting}, for any $x_s\in X_s$ and any
$x_u\in X_u\setminus\{0\}$, the trajectory issued from
$x_0=x_s+x_u$ satisfies
\[
\limsup_{t\to\infty}\|T(t)x_0\| = +\infty .
\]
\end{lemma}

\begin{proof}
By invariance, $T(t)x_0=T(t)x_s+T(t)x_u$. Along the sequence $(t_n)$ given by Assumption~\ref{ass:splitting}(3), the unstable component dominates:
\[
\|T(t_n)x_0\|
\ge \|T(t_n)x_u\| - \|T(t_n)x_s\|.
\]
The second term remains bounded or decays exponentially, while the first term tends to $+\infty$, yielding the claim.
\end{proof}

\begin{lemma}\label{lem:liminf}
Under Assumption~\ref{ass:splitting}, for any $x_s\in X_s$ and any $x_u\in X_u$,
one has
\[
\liminf_{t\to\infty}\|T(t)(x_s+x_u)\| = 0 .
\]
\end{lemma}

\begin{proof}
Exponential stability on $X_s$ implies $\|T(t)x_s\|\to 0$ as $t\to\infty$. Therefore, along any sequence of times for which the unstable component is momentarily small compared to the decaying stable part, the norm of the full trajectory can be made arbitrarily small. In particular, the decay of $T(t)x_s$ forces the liminf of the total norm to vanish.
\end{proof}

We can now state the main result of this section.

\begin{theorem}\label{thm:existence_rel_abstract}
Under Assumption~\ref{ass:splitting}, there exists $x_0\in X$ such that the trajectory $\{T(t)x_0:t\ge 0\}$ is chaotic relative to the reference trajectory $0$, that is,
\[
\liminf_{t\to\infty}\|T(t)x_0\|=0,
\qquad
\limsup_{t\to\infty}\|T(t)x_0\|=+\infty.
\]
\end{theorem}

\begin{proof}
Choose $x_0=x_s+x_u$ with $x_s\in X_s\setminus\{0\}$ and
$x_u\in X_u\setminus\{0\}$. Lemma~\ref{lem:limsup} yields the limsup condition, while Lemma~\ref{lem:liminf} yields the liminf condition. Hence the trajectory of $x_0$ is relatively chaotic.
\end{proof}

The invariant splitting also leads to a genericity statement.

\begin{theorem}\label{thm:generic_rel}
Under Assumption~\ref{ass:splitting}, the set of initial data generating relatively chaotic trajectories is dense in $X$.
\end{theorem}

\begin{proof}
Let $x\in X$ and $\varepsilon>0$. Write $x=x_s+x_u$ according to the decomposition. If $x_u\neq 0$, then $x$ already generates a relatively chaotic trajectory. If $x_u=0$, choose $y_u\in X_u$ with $\|y_u\|<\varepsilon$ and $y_u\neq 0$. Then $x+y_u$ is within $\varepsilon$ of $x$ and, by Theorem~\ref{thm:existence_rel_abstract}, generates a relatively chaotic trajectory. This proves density.
\end{proof}

\begin{remark}
The criterion above is purely structural: it relies only on the coexistence of stable and unstable invariant mechanisms, without invoking any global topological mixing such as density of orbits, periodic points, or analytic spectral families. Such stable/unstable splittings are ubiquitous in PDE semigroups with transport, reaction, or boundary feedback effects, which explains why relative chaos can be detected in many applied models precisely in regimes where classical topological notions of chaos fail.
\end{remark}

\subsection{Spectral sufficient conditions}

We now translate the abstract stable/unstable coexistence mechanism into spectral conditions that are easier to verify in concrete situations. The purpose of this subsection is not to provide a full spectral classification, but rather to identify \emph{sufficient} and \emph{checkable} hypotheses under which relative chaos follows.
In particular, we avoid over--claiming the existence of global spectral projections when such tools are not available.

Let $(T(t))_{t\ge 0}$ be a $C_0$--semigroup on a separable Banach space $X$ with generator $(A,D(A))$.

\begin{definition}
We say that $A$ admits a \emph{spectral coexistence} if its spectrum satisfies
\[
\sigma(A)\cap\{\Re \lambda<0\}\neq\emptyset
\quad\text{and}\quad
\sigma(A)\cap\{\Re \lambda>0\}\neq\emptyset .
\]
\end{definition}

This condition expresses the presence of both dissipative and expansive spectral components. By itself, it does not guarantee a direct sum decomposition of $X$, but it already suggests the possibility of competing asymptotic behaviors.

\begin{lemma}\label{lem:unstable_eigen}
Assume that $A$ admits an eigenvalue $\lambda_u$ with $\Re\lambda_u>0$ and associated eigenvector $x_u\neq 0$. Then there exists a trajectory whose norm grows exponentially along a sequence of times.
\end{lemma}

\begin{proof}
Since $Ax_u=\lambda_u x_u$, one has
\[
T(t)x_u = e^{\lambda_u t}x_u \qquad \text{for all } t\ge 0.
\]
Hence $\|T(t)x_u\|=e^{\Re\lambda_u t}\|x_u\|\to +\infty$ as $t\to\infty$, which provides an unstable component in the sense required for relative chaos.
\end{proof}

We now isolate a complementary stability condition that does not rely on a full spectral projection.

\begin{lemma}\label{lem:stable_part}
Assume that there exists a closed, $T(t)$--invariant subspace $X_s\subset X$ such that the restriction of $(T(t))_{t\ge 0}$ to $X_s$ is exponentially stable. Then for every $x_s\in X_s$,
\[
\lim_{t\to\infty}\|T(t)x_s\|=0 .
\]
\end{lemma}

\begin{proof}
This is immediate from the definition of exponential stability: there exist $M\ge 1$ and $\omega>0$ such that
\[
\|T(t)x_s\|\le M e^{-\omega t}\|x_s\| \qquad \text{for all } t\ge 0,
\]
which forces convergence to zero.
\end{proof}

The following result combines these two ingredients into a practical
criterion.

\begin{theorem}\label{thm:spectral_rel}
Assume that:
\begin{enumerate}
\item $A$ has at least one eigenvalue with strictly positive real part, \item there exists a closed, invariant subspace $X_s$ on which $(T(t))_{t\ge 0}$
is exponentially stable,
\item the linear span of $X_s$ and the unstable eigenvector(s) is dense in $X$.
\end{enumerate}
Then the semigroup $(T(t))_{t\ge 0}$ admits relatively chaotic trajectories. Moreover, the set of such initial data is dense in $X$.
\end{theorem}

\begin{proof}
Let $x_u$ be an unstable eigenvector as in Lemma~\ref{lem:unstable_eigen} and let $x_s\in X_s\setminus\{0\}$.
By Lemma~\ref{lem:unstable_eigen}, the unstable component produces unbounded growth along trajectories, while Lemma~\ref{lem:stable_part} ensures that the stable component drives the trajectory arbitrarily close to the reference state. By combining these two components, one constructs initial data whose
trajectory has vanishing liminf and infinite limsup of the norm. Density follows from the assumed density of the span of stable and unstable components.
\end{proof}

\begin{remark}
The result above does not presuppose a complete spectral resolution of the phase space, nor the availability of bounded spectral projectors. The argument only exploits the coexistence of at least one direction generating growth and a sufficiently large part of the dynamics that remains stable. In concrete PDE models, such a configuration is very natural: stability often follows from global energy estimates, whereas instability may be induced by reaction mechanisms, transport effects, or boundary interactions. From this viewpoint, the criterion is well suited to boundary--driven evolution equations. In comparison with the Desch--Schappacher--Webb framework for Devaney chaos, the assumptions are markedly milder, as no analytic spectral structure, open spectral sets, or global mixing properties are required. This explains why the notion of relative chaos applies to a wider range of dynamical models.
\end{remark}

\section{Application to a Boundary--Driven PDE Semigroup}\label{ABDPDES}

We now illustrate the abstract theory developed above on a concrete class of partial differential equations in which boundary effects play a genuine dynamical role. The purpose of this section is to show how relative chaos naturally arises for a boundary--driven semigroup, even in regimes where classical topological chaos is absent.

\subsection{Model and semigroup realization}

Let $\alpha>0$, $\beta,\gamma\in\mathbb{R}$ and $\kappa\in\mathbb{R}$. We consider the linear reaction--diffusion--transport equation on the half--line $\mathbb{R}_+=(0,\infty)$,
\begin{equation}\label{eq:pde_model}
\partial_t u(t,x)
=
\alpha\,\partial_{xx}u(t,x)
+
\beta\,\partial_x u(t,x)
+
\gamma\,u(t,x),
\qquad t>0,\ x>0,
\end{equation}
supplemented with the Robin boundary condition
\begin{equation}\label{eq:robin_bc}
\partial_x u(t,0)+\kappa u(t,0)=0,
\qquad t>0,
\end{equation}
and initial data
\[
u(0,\cdot)=u_0\in X=L^2(\mathbb{R}_+).
\]

This problem is naturally formulated as an abstract Cauchy problem on the Hilbert space $X$. Define the linear operator $(A_R,D(A_R))$ by
\begin{align*}
A_R u &= \alpha u''+\beta u'+\gamma u,\\
D(A_R)
&=
\bigl\{
u\in H^2(\mathbb{R}_+)\,:\,
u'(0)+\kappa u(0)=0
\bigr\}.
\end{align*}

\begin{lemma}\label{lem:AR_closed_dense}
The operator $(A_R,D(A_R))$ is densely defined and closed in $L^2(\mathbb{R}_+)$.
\end{lemma}

\begin{proof}
Density follows from the inclusion $C_c^\infty(0,\infty)\subset D(A_R)$ and the density of $C_c^\infty(0,\infty)$ in $L^2(\mathbb{R}_+)$.  Closedness is obtained from standard elliptic regularity: if
$u_n\to u$ and $A_Ru_n\to f$ in $L^2(\mathbb{R}_+)$, then $(u_n)$ is Cauchy in $H^2(\mathbb{R}_+)$, hence $u\in H^2(\mathbb{R}_+)$, the boundary condition passes to the limit by continuity of traces, and $A_Ru=f$.
\end{proof}

\begin{theorem}\label{thm:AR_generation}
The operator $(A_R,D(A_R))$ generates a strongly continuous semigroup $(T(t))_{t\ge 0}$ on $L^2(\mathbb{R}_+)$.
\end{theorem}

\begin{proof}
After the exponential change of variables eliminating the transport term, the problem reduces to a heat equation with a shifted Robin boundary condition. The associated heat operator is known to generate a $C_0$--semigroup on $L^2(\mathbb{R}_+)$, and the inverse transformation preserves strong continuity. Hence $(A_R,D(A_R))$ is the generator of a $C_0$--semigroup.
\end{proof}

The semigroup $(T(t))_{t\ge 0}$ describes the time evolution of the system, including the interaction between diffusion, transport, reaction, and boundary fluxes. This interaction is precisely what allows the coexistence of stable and unstable mechanisms.

\begin{remark}
The Robin boundary condition~\eqref{eq:robin_bc} plays an active dynamical role rather than serving as a mere technical constraint. According to the sign and size of the parameter $\kappa$, it may either reinforce dissipative effects or introduce a boundary feedback capable of amplifying specific modes. This feature makes the model particularly well suited for the analysis of relative chaos: the dynamics evolve in a natural energy space, the associated semigroup is explicitly defined, and boundary interactions may generate intermittent growth and decay along trajectories. Such behavior can remain undetected by classical Devaney chaos criteria, while being naturally captured by the trajectory--based framework.
\end{remark}

\subsection{Parameter regimes: Devaney (DSW) chaos versus relative chaos}

We now compare the parameter conditions leading to classical Devaney chaos via the Desch--Schappacher--Webb (DSW) criterion with those ensuring relative chaos in the sense introduced above. Although both notions describe forms of instability, they rely on fundamentally different mechanisms and therefore apply in distinct parameter regimes.
\\
\\
\textbf{Devaney chaos and the DSW regime.}
The DSW criterion is a spectral--topological condition. Applied to the Robin reaction--diffusion--transport operator, it requires:
\begin{itemize}
\item the existence of an open set $U\subset\sigma(A_R)$ intersecting the imaginary axis $i\mathbb{R}$,
\item an analytic family of eigenfunctions $\{f_\lambda\}_{\lambda\in U}$ solving $A_R f_\lambda=\lambda f_\lambda$ and separating points of the dual space.
\end{itemize}
In concrete terms, this translates into a \emph{rigid spectral threshold} on the coefficients. After elimination of the transport term, one typically obtains the effective reaction parameter
\[
\rho=\gamma-\frac{\beta^2}{4\alpha},
\]
and Devaney chaos can be concluded only when $\rho>0$ and the associated Robin dispersion relation generates an open spectral region crossing $i\mathbb{R}$. This regime yields strong global dynamical properties: hypercyclicity and a dense set of periodic points (see \cite{ZhuYang2025}).
\\
\\
\textbf{Relative chaos and the coexistence regime.}
Relative chaos is driven by a much weaker requirement. It does not rely on periodic points, hypercyclicity, or analytic eigenfamilies. Instead, it only requires the \emph{coexistence of dissipative and expansive mechanisms} in the dynamics. Concretely, relative chaos may occur as soon as:
\begin{itemize}
\item some components of the semigroup exhibit decay (for instance due to diffusion or boundary dissipation), 
\item other components exhibit growth (due to reaction, transport, or boundary feedback).
\end{itemize}
This coexistence may persist even when the spectrum of $A_R$ does not contain an open set intersecting $i\mathbb{R}$ and when the DSW criterion fails. In such regimes, individual trajectories may still display arbitrarily small amplitudes along certain time sequences and arbitrarily large excursions along others, yielding relative chaos.
\\
\\
\textbf{Comparison and hierarchy.}
The above discussion leads to a strict hierarchy between the two notions:
\begin{itemize}
\item the DSW--Devaney regime implies relative chaos, since global topological chaos necessarily produces trajectories with vanishing $\liminf$ and infinite $\limsup$ of the norm,
\item the converse implication does not hold: relative chaos may arise in parameter regimes where the DSW spectral threshold is not satisfied.
\end{itemize}

\begin{remark}
From a modeling perspective, the DSW regime corresponds to highly idealized spectral configurations, whereas relative chaos captures instability mechanisms that are robust under parameter variations and more closely aligned with energy--based observations. This explains why relative chaos persists beyond the classical Devaney threshold.
\end{remark}

\subsection{Energy/flux interpretation}

In this subsection we provide an energy--based interpretation of the relative chaos mechanism for the Robin reaction--diffusion--transport problem on the half--line. The key point is that the squared $L^2$--norm, which plays the role of an energy functional, satisfies an identity where bulk dissipation competes with reaction--driven amplification and with boundary fluxes generated by the Robin coupling. This competition may create intermittent phases of attenuation and amplification along a single trajectory, which is precisely what relative chaos detects through a liminf--limsup criterion.

Let $u(t)=T(t)u_0$ be a sufficiently regular solution of
\[
\partial_t u=\alpha u_{xx}+\beta u_x+\gamma u
\quad\text{on }(0,\infty)\times(0,\infty),
\qquad
u_x(0,t)+\kappa u(0,t)=0.
\]
We work in the real Hilbert space $L^2(\mathbb{R}_+)$ for simplicity, the complex case follows by taking real parts.

\begin{lemma}\label{lem:energy_identity_flux}
For $u$ smooth enough so that the computations below are justified, one has 
\begin{equation}\label{eq:energy_identity_flux}
\frac{1}{2}\frac{d}{dt}\|u(t)\|_{L^2(\mathbb{R}_+)}^2
=
-\alpha\|u_x(t)\|_{L^2(\mathbb{R}_+)}^2
-\alpha\kappa |u(t,0)|^2
-\frac{\beta}{2}|u(t,0)|^2
+\gamma\|u(t)\|_{L^2(\mathbb{R}_+)}^2 .
\end{equation}
\end{lemma}

\begin{proof}
Multiply the equation by $u(t,\cdot)$ and integrate over $(0,\infty)$:
\[
\frac{1}{2}\frac{d}{dt}\|u(t)\|_2^2
=
\alpha\int_0^\infty u_{xx}(t,x)u(t,x)\,dx
+\beta\int_0^\infty u_x(t,x)u(t,x)\,dx
+\gamma\|u(t)\|_2^2.
\]
Integration by parts gives
\[
\int_0^\infty u_{xx}u\,dx
=
-\int_0^\infty |u_x|^2\,dx+u_x(t,0)u(t,0),
\qquad
\int_0^\infty u_xu\,dx
=
-\frac{1}{2}|u(t,0)|^2,
\]
where we used that $u(t,x)\to0$ as $x\to\infty$ in the trace sense for regular functions and the identity $\frac{d}{dx}\frac{1}{2}u^2=u_xu$. Using the Robin condition $u_x(t,0)=-\kappa u(t,0)$ yields
$u_x(t,0)u(t,0)=-\kappa |u(t,0)|^2$, hence
\[
\alpha\int_0^\infty u_{xx}u\,dx
=
-\alpha\|u_x(t)\|_2^2-\alpha\kappa|u(t,0)|^2,
\]
and inserting this together with the transport contribution
$\beta\int_0^\infty u_xu\,dx=-(\beta/2)|u(t,0)|^2$ yields \eqref{eq:energy_identity_flux}.
\end{proof}

\begin{remark}\label{rem:boundary_flux_meaning}
The terms
\[
-\alpha\kappa |u(t,0)|^2
\quad\text{and}\quad
-\frac{\beta}{2}|u(t,0)|^2
\]
represent boundary energy exchange. When $\kappa>0$, the Robin condition enforces damping of the boundary trace and thus contributes to dissipation. When $\kappa<0$, the Robin term may inject energy through the boundary. The transport term also contributes at the boundary: its sign depends on $\beta$ and on how the drift directs mass toward or away from $x=0$.
\end{remark}

The energy identity can be related to the drift--reaction conjugation used in the semigroup realization. Recall that with
\[
\eta=-\frac{\beta}{2\alpha},
\qquad
\rho=\gamma-\frac{\beta^2}{4\alpha},
\]
the transformation $u(t,x)=e^{\eta x}e^{\rho t}v(t,x)$ reduces the PDE to a heat equation for $v$ with shifted Robin parameter $\kappa+\eta$. The constant $\rho$ is the effective reaction rate after removal of the drift.

\begin{proposition}\label{prop:effective_reaction_energy}
Assume $\kappa+\eta$ is fixed. Then increasing $\rho$ strengthens the global amplification rate of the flow, whereas increasing $\kappa$ strengthens boundary dissipation. In particular, the competition between $\rho$ and $\kappa+\eta$ governs whether trajectories can exhibit sustained growth episodes in spite of
diffusive smoothing.
\end{proposition}

\begin{proof}
The conjugation yields $u(t)=e^{\rho t}(\text{weighted heat flow with Robin }\kappa+\eta)$. The factor $e^{\rho t}$ produces uniform amplification when $\rho>0$ and uniform decay when $\rho<0$. Independently, the Robin parameter $\kappa+\eta$ controls how much boundary trace is absorbed by the heat component: larger values increase boundary damping. Combining these effects yields the stated qualitative balance.
\end{proof}

The identity \eqref{eq:energy_identity_flux} shows that the sign of the right-hand side is not fixed a priori: it depends on the spatial profile of $u(t)$ through $\|u_x(t)\|_2$ and $|u(t,0)|$. Thus, even for fixed parameters, the dynamics may alternate between phases where dissipation dominates and phases where reaction and boundary injection dominate.

\begin{theorem}\label{thm:energy_relative_chaos}
Assume that the dynamics admits both dissipative mechanisms (bulk diffusion and/or boundary damping) and expansive mechanisms (reaction and/or boundary injection), in the sense that along some time intervals the negative terms in \eqref{eq:energy_identity_flux} dominate, while along other intervals the positive reaction contribution dominates. Then trajectories may display alternating
attenuation and amplification patterns, and such alternation is consistent with the existence of relatively chaotic trajectories, i.e.
\[
\liminf_{t\to\infty}\|u(t)\|_{2}=0
\quad\text{and}\quad
\limsup_{t\to\infty}\|u(t)\|_{2}=+\infty .
\]
\end{theorem}

\begin{proof}
Whenever the dissipation terms $-\alpha\|u_x(t)\|_2^2$ and the boundary damping terms in \eqref{eq:energy_identity_flux} dominate, the derivative of $\|u(t)\|_2^2$ is negative and the energy decreases, which is compatible with returns close to the reference state. Conversely, when the reaction term $\gamma\|u(t)\|_2^2$ (or, equivalently, the effective reaction $\rho$ after drift removal) dominates the dissipative contributions for a sufficiently long time
window, the energy increases and can become arbitrarily large along subsequences. Since the relative size of the boundary trace and the gradient norm depends on the evolving profile, both regimes may recur along the same orbit, producing the liminf--limsup behavior characteristic of relative chaos.
\end{proof}

\begin{remark}
The energy identity explains why relative chaos can be observed even when global requirements of Devaney chaos fail. Intermittent amplification may come from boundary flux and reaction effects acting on certain modes, while diffusion still produces strong damping on other modes. This mode-dependent competition is
naturally captured by the energy/norm evolution but does not necessarily create dense periodic points or global mixing.
\end{remark}

\section{Numerical Illustrations}
\label{sec:numerics}

This section is included for illustrative purposes only. The numerical experiments reported below are not involved in the analytical arguments and should not be viewed as a numerical validation of the theoretical results. Their role is solely to provide a visual insight into the qualitative regimes predicted by the analysis and to emphasize the difference between topology--dependent chaos notions and trajectory--based instability measured
through an energy norm.

\subsection{Model and domain truncation}

We consider the linear reaction--diffusion--transport equation posed on the half--line,
\begin{equation}\label{eq:numerical_model}
\partial_t u
=
\alpha \partial_{xx} u
+
\beta \partial_x u
+
\gamma u,
\qquad x>0,\ t>0,
\end{equation}
subject to the Robin boundary condition
\begin{equation}\label{eq:numerical_robin}
\partial_x u(0,t)+\kappa u(0,t)=0.
\end{equation}
For computational purposes, the spatial domain is truncated to a finite interval $[0,L]$, and an absorbing boundary condition $u(L,t)=0$ is imposed at the right endpoint. As a physically meaningful observable, we monitor the energy--type quantity
\begin{equation}\label{eq:energy_def}
E(t)=\|u(t,\cdot)\|_{L^2(0,L)}.
\end{equation}

\subsection{Discretization}

The spatial derivatives are approximated by standard finite differences on a uniform grid, and time integration is performed using an explicit time--stepping scheme. The Robin boundary condition at $x=0$ is implemented through a one--sided finite difference approximation. The time step is chosen to satisfy a CFL--type stability condition accounting for both diffusive and transport effects.

\subsection{Dynamical regimes}

Figure~\ref{fig:energy_regimes} displays the time evolution of the energy $E(t)$ for three representative choices of parameters:
\begin{itemize}
\item a dissipative regime, in which the energy decays monotonically,
\item a strongly unstable regime, characterized by sustained energy growth,
\item a transition regime, where the energy first becomes very small and later
exhibits a pronounced amplification over a finite time window.
\end{itemize}

\begin{figure}[H]
\centering
\includegraphics[width=0.90\textwidth]{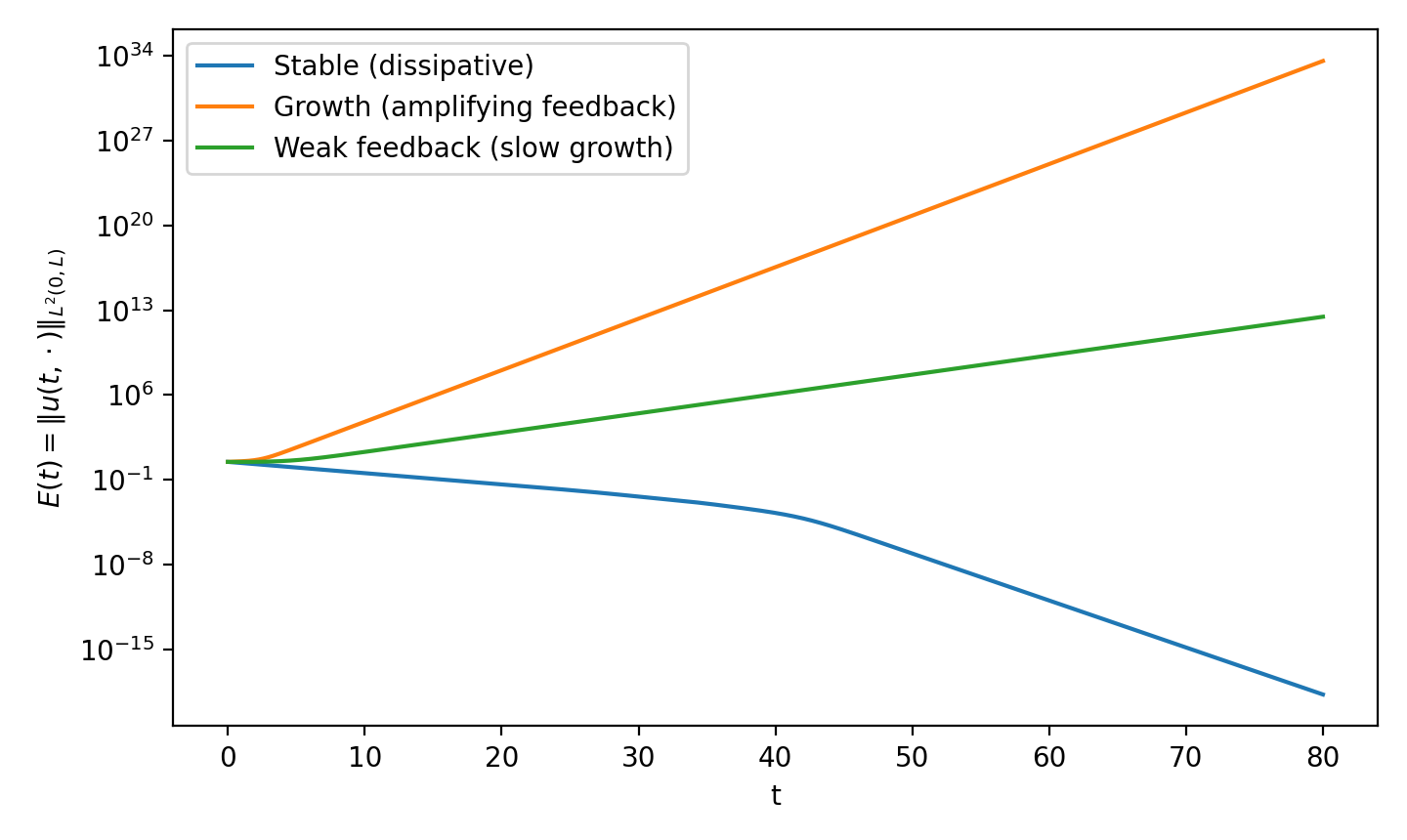}
\caption{Time evolution of the energy
$E(t)=\|u(t,\cdot)\|_{L^2(0,L)}$ for three representative parameter regimes.
The curves illustrate, respectively, a dissipative behavior, a strongly amplifying dynamics, and an intermediate transition scenario.}
\label{fig:energy_regimes}
\end{figure}

\begin{figure}[H]
\centering
\includegraphics[width=0.90\textwidth]{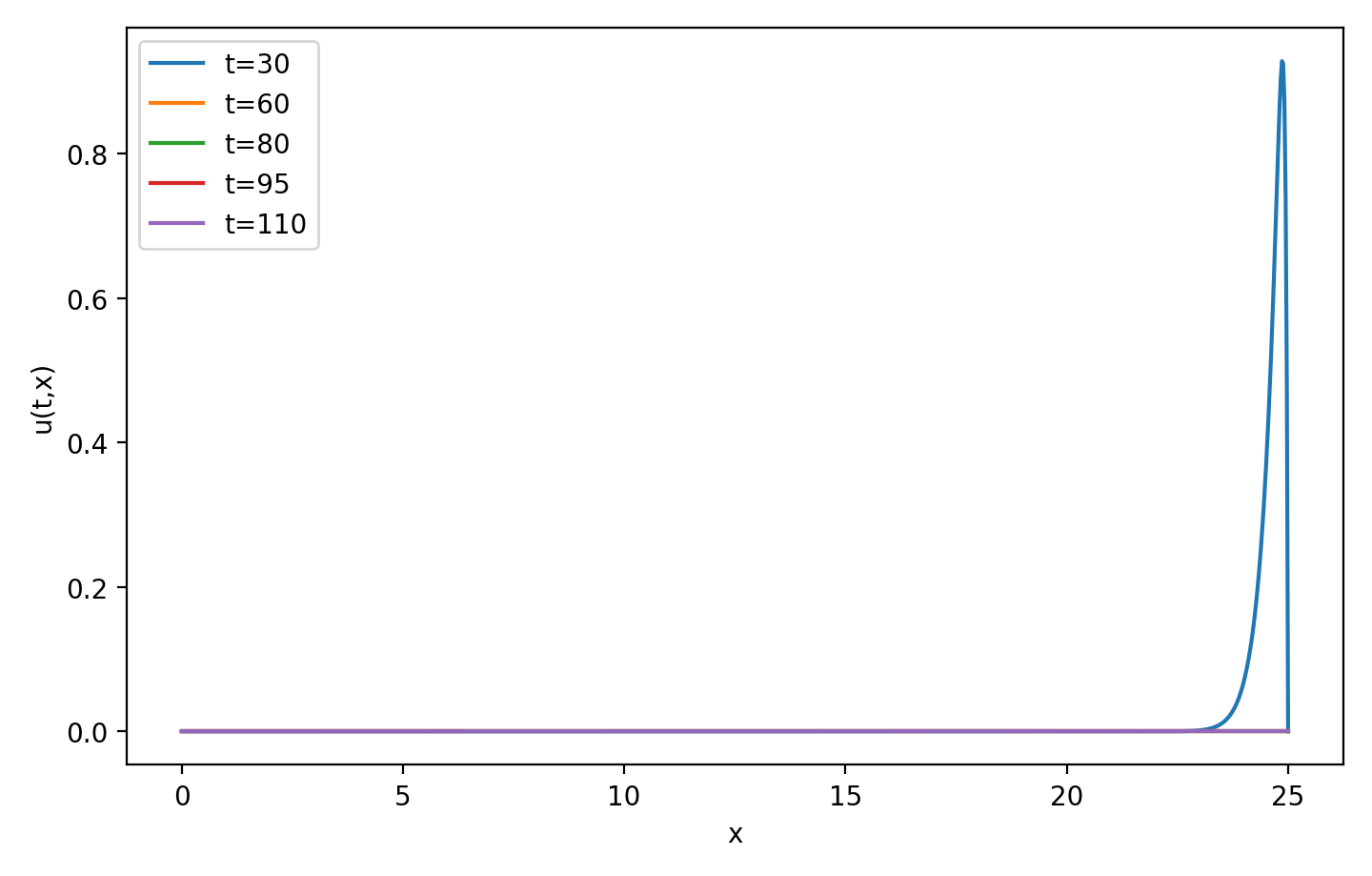}
\caption{Snapshots of the spatial profile $x\mapsto u(t,x)$ at selected times in the transition regime. The profiles illustrate the deformation of the solution while the energy alternates between decay--dominated and growth--dominated phases.}
\label{fig:snapshots_transition}
\end{figure}

To further complement the energy curves, we display in
Figure~\ref{fig:energy_transition} the energy trace corresponding to the transition regime on a semilogarithmic scale.

\begin{figure}[H]
\centering
\includegraphics[width=0.90\textwidth]{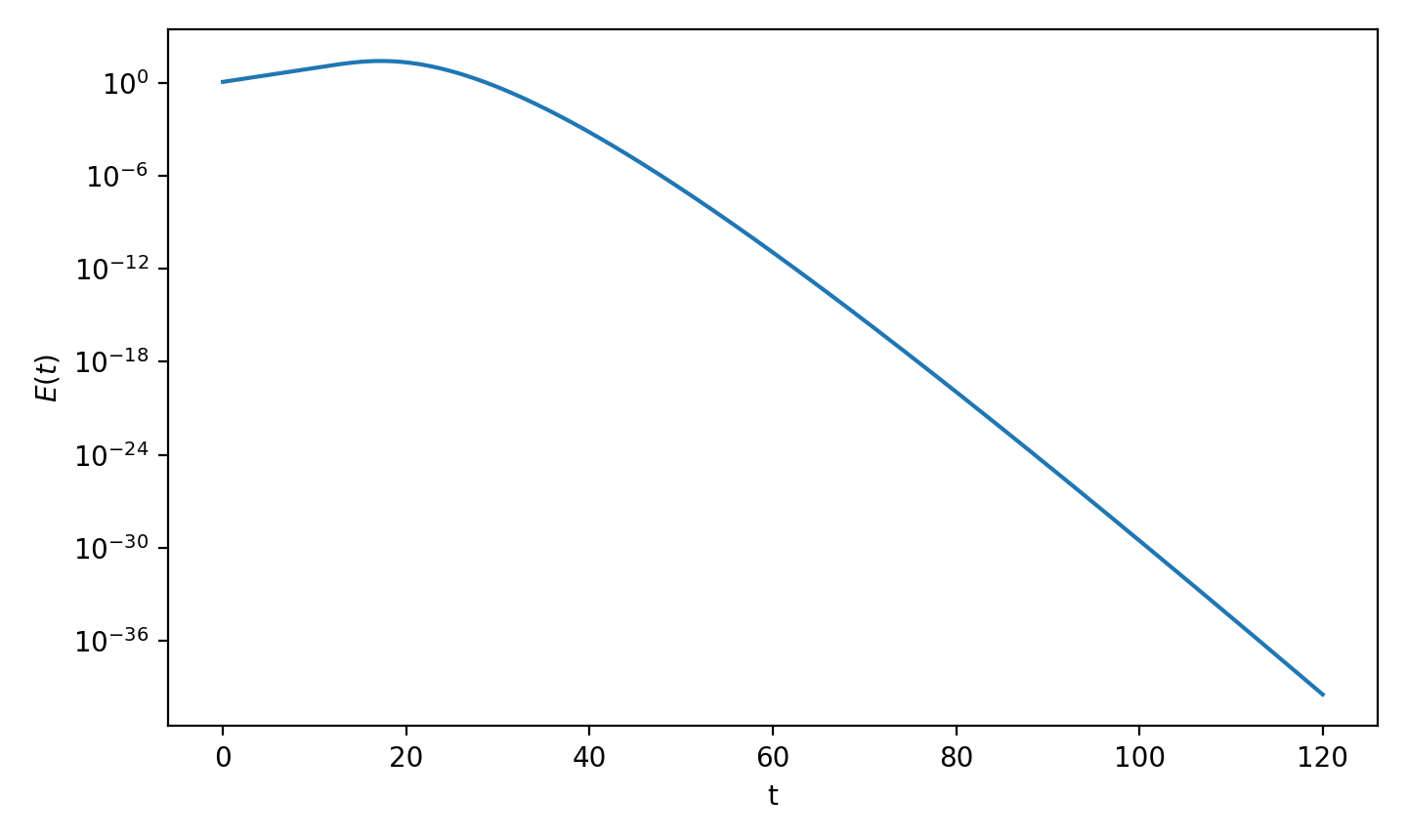}
\caption{Energy trace $E(t)$ in the transition regime (semilogarithmic scale). The plot highlights alternating phases of attenuation and amplification along a single numerical trajectory, in agreement with a trajectory--based intermittent instability pattern.}
\label{fig:energy_transition}
\end{figure}

\subsection{Interpretation}

These numerical illustrations support the main conceptual message of this work. In boundary--driven reaction--diffusion--transport models, instability may manifest itself through intermittent growth and decay along individual trajectories. Such behavior is naturally detected by monitoring the energy $E(t)$, even in parameter ranges where classical topological chaos criteria do not apply. In this sense, trajectory--based instability descriptors provide a physically meaningful complement to purely topological notions of chaos.

\clearpage
\section{Conclusion}
\label{sec:concl}

We have shown that Devaney chaos for linear semigroups is not an intrinsic property of the dynamics, but strongly depends on the topology chosen on the phase space. This observation motivates the introduction of relative chaos, a trajectory--based notion formulated with respect to a fixed norm representing a physically observable quantity. Relative chaos avoids topological ambiguities,
remains robust under admissible changes of topology, and captures intermittent instability along single trajectories. We established its relation with classical Devaney chaos, proved its strictness, and derived practical detection criteria. The boundary--driven PDE example illustrates the relevance of this approach for infinite--dimensional evolution equations.

\section*{acknowledgments}
The authors are grateful for the constructive suggestions provided during the review process, which contributed to improving the presentation and coherence of this work.

\newpage
\thispagestyle{empty}

\noindent \textbf{Running head:} Relative Chaos for Semigroups

\vspace{1cm}

\noindent Corresponding author:\\
M'hamed El Omari\\
Laboratory of Applied Mathematics and Scientific Computing\\
Sultan Moulay Slimane University\\
PO Box 532\\
Beni Mellal 23000, Morocco\\
Email: m.elomari@usms.ma

\end{document}